\documentclass{article}[12pt]
\usepackage{amsmath}
\usepackage{amssymb}
\usepackage{extarrows}
\usepackage{hyperref}
\newtheorem{theorem}{Theorem}

\newtheorem{lemma}[theorem]{Lemma}

\newtheorem{proposition}[theorem]{Proposition}
\newtheorem{remark}[theorem]{Remark}

\newenvironment{proof}[1][Proof]{\noindent\textbf{#1.} }{\ \rule{0.5em}{0.5em}}
\usepackage{amssymb}
\usepackage{graphicx}
\usepackage[dvipsnames,usenames]{color}

\input{epsf.tex}
\usepackage[affil-it]{authblk}
\author{Ioannis Dimitriou \footnote{idimit@uoi.gr}\footnote{Corresponding author.}}
\affil{\small Department of Mathematics, 
	University of Ioannina, 
	45110, Ioannina, Greece.}
\begin{document}
\title{On Markov-dependent reflected autoregressive processes and related models}

\maketitle
\begin{abstract}
   In this paper, we study Markov-dependent reflected autoregressive processes, and other related models the analysis of which results in a vector-valued fixed-point functional equation of a certain type. In queueing terms, such processes describe the workload just before a customer arrival, which makes obsolete a fraction of the work already present, and where the interarrival time and the service time depend on a common discrete time Markov chain. Our primary aim is to derive the Laplace-Stieltjes transform vector of the steady-state workload via a recursive approach. We consider the case where given the state of the underlying Markov chain, the interarrival time and the service time are conditionally independent. Moreover, we further focus on the case where there is also additional dependence based on the Farlie-Gumbel-Morgenstern copula, as well as the case where there is a dependence based on a class of multivariate matrix-exponential distributions. The transient analysis of the Markov-modulated reflected autoregressive process with a more general dependence structure is also investigated. Finally, motivated by queueing applications, we consider two other Markov-dependent models that are described by a similar stochastic recursion: the modulated shot-noise single server queue, and the single-server queue with service time randomly dependent on the waiting time. 
    \end{abstract}
    \vspace{2mm}
	
	\noindent
	\textbf{Keywords}: {Workload; Laplace-Stieltjes transform; Recursion; Markov-modulation}

\section{Introduction}
Recently, much attention has been devoted to the analysis of reflected autoregressive processes. In the simplest scenario, such processes are described by recursions of the form $W_{n+1}=[aW_{n}+S_{n}-A_{n+1}]^{+}$ where $a\in (0,1)$ and $\{S_{n}\}_{n\in\mathbb{N}_{0}}$, $\{A_{n}\}_{n\in\mathbb{N}}$ are sequences of independent and identically distributed (i.i.d.) random variables, and independent of each other. The authors in \cite{box1} provided an extensive analysis of such a process. Several generalizations of the work in \cite{box1} where presented in \cite{box2,box3,hoo}. The restriction of these works was that $\{S_{n}\}_{n\in\mathbb{N}_{0}}$, $\{A_{n}\}_{n\in\mathbb{N}}$ were mutually \textit{independent} sequences of i.i.d. random variables. Quite recently, in \cite{dimi} the authors generalized the work in \cite{box1} by considering, among others, non-trivial dependence structures among $\{S_{n}\}_{n\in\mathbb{N}_{0}}$, $\{A_{n}\}_{n\in\mathbb{N}}$.

In this work, we focus on the case where the sequences $\{S_{n}\}_{n\in\mathbb{N}_{0}}$, $\{A_{n}\}_{n\in\mathbb{N}}$ depend on a common discrete time Markov chain. In particular, we assume that each transition of the Markov chain generates a new interarrival time $A_{n+1}$ and its corresponding service time $S_{n}$, thus, we focus on the Markov-dependent version of the process analysed in \cite{box1}. Note that the specific case of $a = 1$ corresponds to the waiting time in a single server queue with Markov-dependent interarrival and service times studied in \cite{adan2}. Markov-dependent structure of the form considered in \cite{adan2} has also been used in insurance mathematics; see \cite{albbox}. The process analysed in \cite{adan2} (i.e., for $a=1$) is a special case of the class of processes studied in \cite{asmkella}, although in \cite{adan2}, all the results were given explicitly. The case of $a=-1$ was investigated in \cite{vlasioudep} (see also \cite[Chapter 5]{vlasiou}) in the context of carousel models. In this work, we focus on the case where $a\in(0,1)$. In queueing terms, the modulated version of the recursion $W_{n+1}=[aW_{n}+S_{n}-A_{n+1}]^{+}$ where $a\in (0,1)$ refers to a generalization of the MAP/G/1 queue, i.e., it is a MAP/G/1 queue with dependencies between successive service times and between interarrival times and service times, and in which, an arrival makes obsolete a fraction of the work already present. 

Note that reflected autoregressive processes described by recursions of the form $W_{n+1}=[aW_{n}+S_{n}-A_{n+1}]^{+}$, where $a\in (0,1)$ are directly connected with queueing systems, since they model the notion of impatience or work removal due to an arrival, i.e., an arrival adds new work, but also causes also a work removal; see \cite{bouboxsig,boubox,jainsig} for related results on queues with work removal. In particular, let $W_{n}$, denote the amount
of work in the system found by customer $n$. Then, the amount of work, $W_{n+1}$ found
by the next arriving customer will be equal to $W_{n+1}=[W_{n}+S_{n}-A_{n+1}-D_{n+1}]^{+}$, where $D_{n+1}$ is the amount of work destroyed by the arrival of the $n+1$ customer. In our case $D_{n+1}:=(1-a)W_{n}$, i.e., a portion of the work found upon his/her arrival.   

Most queueing models assume that interarrival and service
times are independent and identically distributed, and represent
each as a renewal process. This leads to analytically tractable models that are often poor representations of realistic applications where correlations/dependencies between service and interarrival times do exist, such as in many service, telecommunication, manufacturing systems and in ruin theory; see e.g. \cite{adan2,albbox,cive,wang,wu1,wu2} and in \cite[Chapters VII, IX, XIII]{asmalb}. Thus, it is of highest importance to incorporate dependencies in our building models. However, to our best knowledge, there is a lack of progress on analytical results regarding Markov-dependent reflected autoregressive processes, and other related models the analysis of which results in a fixed-point vector-valued functional equation of a certain form.

To fill this gap, our primary motivation is to consider vector-valued generalizations of the work in \cite{box1}. More importantly, our goal is to consider advanced semi-Markovian dependence structures on problems described by reflected autoregressive processes and related models, the analysis of which results in similar vector-valued fixed point functional equations.
\subsection{Our contribution}
The main contributions of the article are the following. 
\begin{enumerate}
    \item To the best of the author's knowledge, this paper is the first in which recursions of the form $W_{n+1}=[aW_{n}+S_{n}-A_{n+1}]^{+}$, where $a\in (0,1)$ are studied under a semi-Markov dependent framework introduced in \cite{adan2}. We also illustrate with numerical results the influence of auto-correlation
and cross-correlation of the interarrivals and service times on the mean workload.
\item In addition to the semi-Markov dependence structure used in \cite{adan2,vlasioudep}, in which given the state of the Markov chain at times $n$, $n+1$, the distributions of $A_{n+1}$, $S_{n}$ are independent of one another for all $n$ (although their distributions depend on the state of the Markov chain), we further consider (for $a\in(0,1)$) the case where there is an additional dependence among  $\{S_{n}\}_{n\in\mathbb{N}_{0}}$, $\{A_{n}\}_{n\in\mathbb{N}}$ based on Farlie-Gumbel-Morgenstern (FGM) copula. More precisely, $\{(S_{n},A_{n+1})\}_{n\in\mathbb{N}_{0}}$ forms a sequence of i.i.d. random vectors with a distribution function defined by the FGM copula, and depends also on the state of the underlying discrete time Markov chain. Among a quite rich pool of copulas (see e.g., \cite{nelsen}), we chose the FGM family due to its polynomial structure that leads to tractable and analytic results. To our best knowledge there are no analytic results referring to the derivation of the Laplace-Stieltjes (LST) transform of the steady-state workload in a Markov-dependent single server queue where interarrival times and service times are also dependent based on copulas (we just mention \cite{trapa,wang,wang1} where either moments or moment approximations of steady-state workload were derived). We also numerically illustrate the effect of the dependence via FGM copula on the mean workload. We observed that the higher the value of the dependence parameter, the lower the mean workload. Moreover, it seems that the dependence parameter $\theta$ can be used as a tool for smoothing cross-correlation, for changing its sign, or even as a tool for vanishing cross-correlation (see Section \ref{num}). We also numerically illustrated the effect of the autoregressive parameter $a$ on the number of iterations that we need in order to provide numerical examples. The more we increase $a$, the more product form terms are required to derive numerical results.
\item We also consider the case where we model the dependence structure of $\{(S_{n},A_{n+1})\}_{n\in\mathbb{N}_{0}}$ by using a bivariate matrix-exponential distribution, which is further dependent on the state of the underlying discrete time Markov chain, and thus, we considerably generalize the results in \cite{bad}, in which the scalar case with $a=1$ was studied.
\item The time-dependent analysis of a Markov-modulated reflected autoregressive process with a quite general dependence structure is also investigated. In particular, we investigate the "autoregressive" analogue of the work in \cite{regte}, in which the authors developed a matrix Wiener-Hopf approach to study an M/G/1 queue with Markov-modulated arrivals and service requirements. We firstly consider the case where interarrival times and service times depend on the state of an underlying finite state Markov chain. Secondly, on top of that, we further consider the case where the interarrival time depends also on the length of the previous service time. %In queueing terms, the modulated version of the recursion $W_{n+1}=[aW_{n}+S_{n}-A_{n+1}]^{+}$ where $a\in (0,1)$ refers to a generalization of the MAP/G/1 queue, i.e., it is a MAP/G/1 type reflected autoregressive process by also allowing dependencies between successive service times and between interarrival times and service times. 
%Note that reflected autoregressive processes described by recursions of the form $W_{n+1}=[aW_{n}+S_{n}-A_{n+1}]^{+}$, where $a\in (0,1)$ are directly connected with queueing systems, since they model the notion of impatience or work removal due to an arrival, i.e., an arrival adds new work, but also causes also a work removal; see \cite{bouboxsig,boubox,jainsig} for related results on queues with work removal. In particular, let $W_{n}$, denote the amount
%of work in the system found by customer $n$. Then, the amount of work, $W_{n+1}$ found
%by the next arriving customer will be equal to $W_{n+1}=[W_{n}+S_{n}-A_{n+1}-D_{n+1}]^{+}$, where $D_{n+1}$ is the amount of work destroyed by the arrival of the $n+1$ customer. In our case $D_{n+1}:=(1-a)W_{n}$, i.e., a portion of the work found upon his/her arrival.   
\item We also investigate two other Markov-dependent models the analysis of which results in vector-valued functional equations of a form that allows a similar solution method as the one considered in the above discussed models. In particular, we first focus on a Markov-dependent queue where the server's speed is workload proportional, i.e., a modulated shot-noise queue. In such a case, we are dealing with a modulated stochastic recursion of the form $W_{n+1}=[e^{-rA_{n+1}}(W_{n}+S_{n})+C_{n+1}]^{+}$, $r>0$, where the distributions of $S_{n}$, $C_{n+1}$ depend on the state of a Markov chain at times $n$, $n+1$. For recent results on shot-noise queueing models, see \cite{shots}. Then, we focus on the analysis of a single-server queue where the service times are dependent on the waiting time, i.e., we are dealing with a modulated stochastic recursion of the form $W_{n+1}=[W_{n}+[S_{n}-cW_{n}]^{+}-A_{n+1}]^{+}$, $c>0$, where $\{S_{n}\}_{n\in\mathbb{N}}$ is an independent sequence of $\exp(\mu)$ distributed random variables. In both cases the Laplace-Stieltjes transform of the steady-state workload is derived via a recursive approach. 
\end{enumerate}

In all the above mentioned models we arrive at a vector-valued fixed point functional equation for the \textit{transform vector} $\tilde{Z}(s)$ of the steady-state workload of the following form:
\begin{equation}
    \tilde{Z}(s)=H(s)\tilde{Z}(\zeta(s))+\tilde{V}(s),\,Re(s)\geq 0,\label{ji1}
\end{equation}
where $H(s)$, $\tilde{V}(s)$ are known matrix, and vector functions, respectively. In the setting of Sections \ref{sec1}-\ref{bila} $\zeta(s)=as$, while in the setting of Section \ref{related}, $\zeta(s)=se^{-rt}$, $r,t>0$ (subsection \ref{shot}), and $\zeta(s)=s+\mu c$ (subsection \ref{deps}). Finally, the analysis of the model in Section \ref{mm} results in the following vector-valued  fixed point functional equation:
\begin{equation}
        Z(r,s,\eta)=G(r,s,\eta)Z(r,as,\eta)+K(r,s,\eta),\,Re(s)\geq 0,\,Re(\eta)\geq 0,|r|<1,\label{ji2}
\end{equation}
where $G(r,s,\eta)$, $K(r,s,\eta)$ are known matrix, and vector functions, respectively. 

% A special case of Markov-dependent inter-arrival and service times is the model with strictly periodic arrivals. Such kind of models arise in the modelling of inventory systems using periodic ordering policies, and in which a new arrival causes the ; see [36,37]. Queueing models with periodic arrival processes have been studied in, e.g., [10,25,32,33].
\subsection{Structure of the paper}
The paper is organized as follows. In Section \ref{sec1} we present the basic model by studying the recursion $W_{n+1}=[aW_{n}+S_{n}-A_{n+1}]^{+}$, $a\in (0,1)$ and where the distributions of $S_{n}$, $A_{n+1}$ depend on the state of a Markov chain at times $n$, $n+1$. We assume both the case where interarrival times are exponentially distributed and the case where it is of phase-type. In Section \ref{fgm}, we further assume that given the state of a Markov chain at times $n$, $n+1$, the sequences $\{S_{n}\}_{n\in\mathbb{N}}$, $\{A_{n}\}_{n\in\mathbb{N}}$ are dependent based on an FGM copula. Section \ref{bila} is devoted to the case where $X_{n}:=S_{n}-A_{n+1}$ follows a bilateral matrix exponential distribution that is dependent on the state of the underlying discrete time Markov chain. In all above cases we focus on the stationary behaviour of the workload process. In Section \ref{mm}, we analyze the time-dependent Markov-modulated version of the reflected autoregressive process and focus on the derivation of the Laplace-Stieltjes transform of the transient distribution of the workload. In subsection \ref{sub}, we considered an even more general dependence structure, where interarrival time depends also on the length of the previous service time. Section \ref{related} is devoted to the analysis of two Markov-dependent models that can be analyzed with a similar solution method: The analysis of a modulated shot-noise queue, i.e., the server's speed is workload proportional, is considered in subsection \ref{shot}. Subsection \ref{deps} is devoted to the analysis of a modulated Markovian single server queue in which the service time of a customer depends on his/her waiting time. Two simple numerical examples are presented in Section \ref{num} to investigate the impact of auto-correlation and cross-correlation of the interarrivals and service times on the mean workload based on the derivations of Section \ref{sec1}. We further numerically investigate the impact of the dependence parameter $\theta$ of the FGM copula on the mean workload, as well as on the cross-correlation based on the derivations of Section \ref{fgm}. The impact of the autoregressive parameter $a$ on the number of iterations we need in order to derive numerical results is also discussed. Concluding remarks and suggestions for future research are presented in Section \ref{con}.
\section{Basic model and analysis}\label{sec1}
In the following, we define the main setting and study the case where the interarrival times and the service
times depend on a common discrete-time Markov chain. Such a model allows dependencies between interarrival and service times. 

Consider a FIFO single-server queue, and let $T_{n}$ be the $n$th arrival to the system with $T_{1}=0$. Define also $A_{n}=T_{n}-T_{n-1}$, $n=2,3,\ldots$, i.e., is the time between the $n$th and $(n-1)$th arrival. Let $S_{n}$ be the service time of the $n$th arrival, $n\geq 1$. We assume that the sequences $\{A_n\}_{n\in\mathbb{N}}$ and
$\{S_n\}_{n\in\mathbb{N}_{0}}$ are autocorrelated as well as cross-correlated. The distributions of the interarrival and service times are regulated by an irreducible discrete-time Markov chain $\{Y_{n}, n\geq 0\}$ with state space $E=\{1, 2,...,N\}$, one-step transition probability matrix $P:=(p_{i,j})_{i,j\in E}$, and stationary distribution $\tilde{\pi}:=(\pi_{1},\ldots,\pi_{N})^{T}$. Let $W_{n}$ be the amount of work in the system just before the $n$th customer
arrival. Such an arrival adds $S_{n}$ work, but also makes obsolete a fixed fraction $1 - a$ of
the work that is already present in the system. Thus, we are dealing with a recursion of the form $W_{n+1}=[aW_{n}+S_{n}-A_{n+1}]^{+}$. We further assume that for $n\geq 0$, $x,y\geq 0$, $i,j=1,\ldots,N$:
\begin{equation}
    \begin{array}{l}
         P(A_{n+1}\leq x,S_{n}\leq y,Y_{n+1}=j|Y_{n}=i,A_{2},\ldots,A_{n},S_{1},\ldots,S_{n-1},Y_{1},\ldots,Y_{n-1}) \vspace{2mm}\\
        =  P(A_{n+1}\leq x,S_{n}\leq y,Y_{n+1}=j|Y_{n}=i)=p_{i,j}P(A_{n+1}\leq x,S_{n}\leq y|Y_{n}=i,Y_{n+1}=j)\vspace{2mm}\\=p_{i,j}F_{S,i}(y)G_{A,j}(x),
    \end{array}\label{mp1}
\end{equation}
where $F_{S,i}(.)$, $G_{A,j}(.)$ denote the distribution functions of service and interarrival times, given $Y_{n}=i$, $Y_{n+1}=j$, respectively. Note that $A_{n+1}$, $S_{n}$, $Y_{n+1}$ are independent of the past given $Y_{n}$, and $A_{n+1}$, $S_{n}$ are conditionally independent given $Y_{n}$, $Y_{n+1}$.

Denote by $I$ the $N\times N$ identity matrix, and by $M^{T}$ the transpose of a matrix (vector) $M$. Let $Z_{i}^{n}(s):=E\left(e^{-sW_{n}}1_{\{Y_{n}=i\}}\right)$, $Re(s)\geq 0$, $i=1,\ldots,N$, $n\geq 0$, and assuming the limit exists, define $Z_{i}(s)=\lim_{n\to\infty}Z_{i}^{n}(s)$, $i=1,\ldots,N$. Let also $\tilde{Z}(s)=(Z_{1}(s),\ldots,Z_{N}(s))^{T}$, and $B^{*}(s)=diag(\beta_{1}^{*}(s),\ldots,\beta_{N}^{*}(s))$, where $\beta_{i}^{*}(s):=\int_{0}^{\infty}e^{-st}dF_{S,i}(t)$, $i=1,\ldots,N$.
\subsection{The case where interarrival times are exponentially distributed}\label{xc}
In this section, we assume that $G_{A,j}(x)=1-e^{-\lambda_{j}x}$, $j=1,\ldots,N$. Denote also $L(s):=diag(\frac{\lambda_{1}}{\lambda_{1}-s},\ldots,\frac{\lambda_{N}}{\lambda_{N}-s})$, $\Lambda=diag(\lambda_{1},\ldots,\lambda_{N})$. The next theorem provides a system of equations for the unknown functions $Z_{i}(s)$, $i=1,\ldots,N$.
\begin{theorem}\label{th}
The transforms $Z_{j}(s)$, $j=1,\ldots,N$ satisfy the system
\begin{equation}
    Z_{j}(s)=\frac{\lambda_{j}}{\lambda_{j}-s}\sum_{i=1}^{N}p_{i,j}\beta_{i}^{*}(s)Z_{i}(as)-\frac{s}{\lambda_{j}-s}v_{j},\label{op}
\end{equation}
where $v_{j}:=\sum_{i=1}^{N}p_{i,j}\beta_{i}^{*}(\lambda_{j})Z_{i}(a\lambda_{j})$, $j=1,\ldots,N$. Equivalently, in matrix notation, the transform vector $\tilde{Z}(s)$ satisfies
\begin{equation}
    \tilde{Z}(s)=H(s)\tilde{Z}(as)+\tilde{V}(s),
    \label{bhj}
\end{equation}
where $H(s)=L(s)P^{T}B^{*}(s)$, $\tilde{V}(s):=sF(s)\tilde{v}$, $F(s)=-\Lambda^{-1}L(s)$, $\tilde{v}:=(v_{1},\ldots,v_{N})^{T}$. 
\end{theorem}
\begin{proof}
From the recursion $W_{n+1}=[aW_{n}+S_{n}-A_{n+1}]^{+}$ we obtain the following equation for the transforms $Z_{j}^{n+1}(s)$, $j=1,\ldots,N$:
\begin{displaymath}
    \begin{array}{rl}
        Z_{j}^{n+1}(s)= &E\left(e^{-sW_{n+1}}1_{\{Y_{n+1}=j\}}\right)=\sum_{i=1}^{N}P(Y_{n}=i)E\left(e^{-sW_{n+1}}1_{\{Y_{n+1}=j\}}|Y_{n}=i\right)  \vspace{2mm}\\
         =& \sum_{i=1}^{N}P(Y_{n}=i)E\left(e^{-s[aW_{n}+S_{n}-A_{n+1}]^{+}}1_{\{Y_{n+1}=j\}}|Y_{n}=i\right)\vspace{2mm}\\
         =& \sum_{i=1}^{N}P(Y_{n}=i)p_{i,j}E\left(e^{-s[aW_{n}+S_{n}-A_{n+1}]^{+}}|Y_{n+1}=j,Y_{n}=i\right)\vspace{2mm}\\
         =&\sum_{i=1}^{N}P(Y_{n}=i)p_{i,j} \left[E\left(\int_{0}^{aW_{n}+S_{n}}e^{-s(aW_{n}+S_{n}-y)}\lambda_{j}e^{-\lambda_{j}y}dy|Y_{n}=i\right)+E\left(\int_{aW_{n}+S_{n}}^{\infty}\lambda_{j}e^{-\lambda_{j}y}dy|Y_{n}=i\right) \right]\vspace{2mm}\\
         =&\sum_{i=1}^{N}P(Y_{n}=i)p_{i,j}E\left(\lambda_{j}e^{-s(aW_{n}+S_{n})}\int_{0}^{aW_{n}+S_{n}}e^{-(\lambda_{j}-s)y}dy+\int_{aW_{n}+S_{n}}^{\infty}\lambda_{j}e^{-\lambda_{j}y}dy|Y_{n}=i\right)\vspace{2mm}\\
         =&\sum_{i=1}^{N}P(Y_{n}=i)p_{i,j}E\left(\frac{\lambda_{j}}{\lambda_{j}-s}e^{-s(aW_{n}+S_{n})}(1-e^{-(\lambda_{j}-s)(aW_{n}+S_{n})})+e^{-\lambda_{j}(aW_{n}+S_{n})}|Y_{n}=i\right)\vspace{2mm}\\
         =&\sum_{i=1}^{N}P(Y_{n}=i)p_{i,j}E\left(\frac{\lambda_{j}e^{-s(aW_{n}+S_{n})}-se^{-\lambda_{j}(aW_{n}+S_{n})}}{\lambda_{j}-s}|Y_{n}=i\right)\vspace{2mm}\\
         =&\sum_{i=1}^{N}p_{i,j}\left[\frac{\lambda_{j}}{\lambda_{j}-s}Z_{i}^{n}(as)\beta_{i}^{*}(s)-\frac{s}{\lambda_{j}-s}Z_{i}^{n}(a\lambda_{j})\beta_{i}^{*}(\lambda_{j})\right]\vspace{2mm}\\
         =&\frac{\lambda_{j}}{\lambda_{j}-s}\sum_{i=1}^{N}p_{i,j}Z_{i}^{n}(as)\beta_{i}^{*}(s)-\frac{s}{\lambda_{j}-s}\sum_{i=1}^{N}p_{i,j}Z_{i}^{n}(a\lambda_{j})\beta_{i}^{*}(\lambda_{j}).
    \end{array}
\end{displaymath}
Letting $n\to\infty$ so that $Z_{j}^{n}(s)$ tends to $Z_{j}(s)$ we get \eqref{op}. Writing the resulting equations in matrix form we get \eqref{bhj}. 
\end{proof}

After $n-1$ iterations of \eqref{bhj} we get:
\begin{equation}
    \tilde{Z}(s)=s\sum_{k=0}^{n-1}a^{k}\prod_{m=0}^{k-1}H(a^{m}s)F(a^{k}s)\tilde{v}+\prod_{m=0}^{n-1}H(a^{m}s)\tilde{Z}(a^{n}s).\label{iterr1}
\end{equation}

In proving the convergence of \eqref{iterr1} as $n\to\infty$ we use the following Lemma.
\begin{lemma}\label{lemma1}
    For $n=0,1,\ldots,$ $||\prod_{m=0}^{n}H(a^{m}s)||\leq c_{s}(\tilde{U}\tau)^{n}$, for $0\leq Re(s)< min_{\{j=1,\ldots,N\}}\{\lambda_{j}a^{-m}\}$, with $\tau\leq 1$, $c_{s}<\infty$, $\tilde{U}<\infty$, and where $||C||=max_{i,j}|C_{i,j}|$ is a matrix norm of $C\in \mathbb{C}^{N\times N}$.
\end{lemma}
\begin{proof}
Note that $H(a^{m}s)=L(a^{m}s)P^{T}B^{*}(a^{m}s)$, and $L(a^{m}s)\to I$, $B^{*}(a^{m}s)\to I$, as $m\to\infty$. Thus, as $m\to \infty$, $||B^{*}(a^{m}s)||\to 1$, $||L^{*}(a^{m}s)||\to 1$. Clearly, for $Re(s)\geq 0$, $m=0,1,\ldots$, $||B^{*}(a^{m}s)||\leq 1$. $L(a^{m}s)$ is well defined for $Re(s)<min_{\{j=1,\ldots,N\}}\{\lambda_{j}a^{-m}\}$. Keeping $s$ in that region, there exists a $K_{s,m}<\infty$ such that $||L^{*}(a^{m}s)||\leq K_{s,m}$. In particular, this term is close to 1 as $m$ takes large values for any $s$, i.e., for some $M>0$, such that $m\geq M$, $||L^{*}(a^{m}s)||\leq \tilde{U}_{m}<\infty$. Thus, for $m\geq M$, 
\begin{displaymath}
   || H(a^{m}s)||\leq ||L(a^{m}s)||||P^{T}||||B^{*}(a^{m}s)||\leq ||L(a^{m}s)||||P^{T}||\leq \tilde{U}_{m}\tau,
\end{displaymath}
where $\tau:=max_{\{i,j=1,\ldots,N\}}\{p_{i,j}\}$. Hence, for $\tilde{U}=max_{\{m=0,1,\ldots,n\}}\{\tilde{U}_{m}\}$, there exists a $c_{s}<\infty$, such that $||\prod_{m=0}^{n}H(a^{m}s)||\leq c_{s}(\tilde{U}\tau)^{n}$, for all $n=0,1,\ldots$.
\end{proof}

Thus,
\begin{equation}
    \tilde{Z}(s)=s\sum_{n=0}^{\infty}a^{n}\prod_{m=0}^{n-1}H(a^{m}s)F(a^{n}s)\tilde{v}+\lim_{n\to\infty}\prod_{m=0}^{n}H(a^{m}s)\tilde{Z}(a^{n+1}s).\label{iterr}
\end{equation}
\begin{remark}
    Note that the $(i,j)$-element, $i,j=1,...,N,$ of the matrix $H(sa^m)$ is equal to
\begin{displaymath}
p_{j,i}\frac{\lambda_{i}}{\lambda_{i}-sa^{m}}\beta_{j}^{*}(sa^{m}).
\end{displaymath}
Note that in the scalar case, that is, in the non-modulated case, the corresponding product on the right-hand side of \eqref{iterr1} was equal to $\prod_{m=0}^{n-1}\frac{\lambda}{\lambda-sa^{m}}\beta^{*}(sa^{m})$. In our case, the terms of the corresponding product depend on $i,j$ and are multiplied with $p_{j,i}$. 

Now, the $(i,j)$-element of the matrix $\prod_{m=0}^{n}H(sa^m)$ is a linear combination of terms, where each is the product of $(n+1)$ terms of the form $F_{w,k}(sa^{m}):=\frac{\lambda_{k}}{\lambda_{k}-sa^{m}}\beta_{w}^{*}(sa^{m})$ (where $k, w$ depend on $i, j$). The coefficients of these products are also the products of $(n+1)$ elements of the one-step transition probability matrix of the background Markov chain (of course related to the indices $k,w$ mentioned above). For example, the $(i,j)$ element of the matrix $\prod_{m=0}^{1}H(sa^m)$ (i.e., $n=1$) is equal to
\begin{displaymath}
    \sum_{l=1}^{N}p_{l,i}p_{j,l}F_{l,i}(s)F_{j,l}(as).
\end{displaymath}
Since the products $F_{l,i}(s)F_{j,l}(as)$ are bounded (as in the scalar case in \cite{box1}), and the products $p_{l,i}p_{j,l}$ are bounded, so it is any element of the matrix $\prod_{m=0}^{1}H(sa^m)$. In general, the $(i,j)$-element of $\prod_{m=0}^{n}H(sa^m)$ is a linear combination of $n\times N$ terms. Each of these terms is a product of $(n+1)$ terms of the form $F_{w,k}(sa^{m})$. The sum of the coefficients of these terms is equal to the probability that the background Markov chain goes from state $j$ to state $i$ in $(n+1)$ steps, i.e., the $(i,j)$-element of $(P^{T})^{n+1}$ (recall that $P$ is the one-step transition probability matrix of the background Markov chain and $P^{T}$ its transpose matrix), which is bounded by 1. Note that in case we have an ergodic background Markov chain this sum is smaller than 1 (and positive), and as $n\to\infty$ it approaches the limiting probability of state $i$. Now, from the previous discussion and having in mind that as $m\to\infty$ the terms $F_{w,k}(sa^{m})\to 1$, we can realize that the limit in \eqref{iterr} exists.
\end{remark}

Note that as $n\to\infty$, $Z_{j}(a^{n+1}s)\to Z_{j}(0)=E(1_{\{Y=j\}})=P(Y=j)=\pi_{j}$, so that $\tilde{Z}(0)=\tilde{\pi}$. Application of Lemma \ref{lemma1} \eqref{iterr}, yields
\begin{theorem}\label{th1}

\begin{equation}
    \tilde{Z}(s)=s\sum_{n=0}^{\infty}a^{n}\prod_{m=0}^{n-1}H(a^{m}s)F(a^{n}s)\tilde{v}+\prod_{m=0}^{\infty}H(a^{m}s)\tilde{\pi},\,\,\tilde{Z}(0)=\tilde{\pi}.\label{siterr}
\end{equation}    
\end{theorem}
We still have to derive the terms $v_{j}$, $j=1,\ldots,N$, which contains $Z_{i}(a\lambda_{j})$, $i=1,\ldots,N$. 
\begin{proposition}\label{prop1}
    The vector $\tilde{v}$ is the unique solution of the following system of equations:
    \begin{equation}
        v_{j}=e_{j}P^{T}B^{*}(\lambda_{j})[\lambda_{j}\sum_{n=0}^{\infty}a^{n+1}\prod_{m=0}^{n-1}H(a^{m+1}\lambda_{j})F(a^{n+1}\lambda_{j})\tilde{v}+\prod_{m=0}^{\infty}H(a^{m+1}\lambda_{j})\tilde{\pi}],\,j=1,\ldots,N,\label{lpo}
    \end{equation}
    where $e_{j}$ a $1\times N$ vector with the $j$th element equal to one and all the others equal to zero.
\end{proposition}
\begin{proof}
    Setting $s=a\lambda_{j}$ in \eqref{siterr} to obtain, for $j=1,\ldots,N,$
    \begin{displaymath}
        \tilde{Z}(a\lambda_{j})=\lambda_{j}\sum_{n=0}^{\infty}a^{n+1}\prod_{m=0}^{n-1}H(a^{m+1}\lambda_{j})F(a^{n+1}\lambda_{j})\tilde{v}+\prod_{m=0}^{\infty}H(a^{m+1}\lambda_{j})\tilde{\pi}.
    \end{displaymath}
    Simple computations imply that multiplying the above expression with $e_{j}P^{T}B^{*}(\lambda_{j})$ from the left, and having in mind that $v_{j}:=\sum_{i=1}^{N}p_{i,j}\beta_{i}^{*}(\lambda_{j})Z_{i}(a\lambda_{j})$, $j=1,\ldots,N$, yields \eqref{lpo}. The set of equations \eqref{lpo} has a unique positive solution. The uniqueness of the solution follows from the uniqueness of the stationary distribution $\tilde{Z}(s)$, $Re(s)\geq 0$. Solving this linear system of equations we can derive $v_{j}$, $j=1,\ldots,N$.
\end{proof}
\paragraph{Steady-state moments} Once $\tilde{v}$ is known, the transform vector $\tilde{Z}(s)$ is fully specified. Let $M_{i}:=-Z_{i}^{\prime}(0)=E(W_n 1_{\{Y_n=i\}})$, $i=1,\ldots,N$ and $\tilde{M}:=(M_{1},\ldots,M_{N})^{T}$. Differentiating \eqref{bhj} with respect to $s$ and letting $s\to 0$ we obtain the mean workload vector
\begin{displaymath}
    \tilde{M}=(P^{T}a-I)^{-1}(\Phi\tilde{\pi}-\Lambda^{-1}\tilde{v}),
\end{displaymath}
where $\Phi:=L^{\prime}(0)P^{T}+P^{T}B^{*\prime}(0)$, and $I$ be the $N\times N$ identity matrix.
\begin{remark}
    It would be of high importance to consider also recursions that result in the following matrix functional equations:
\begin{equation}
        \tilde{Z}(s)=H(s)\tilde{Z}(\zeta(s))+\tilde{V}(s).\label{aqw}
    \end{equation}
    In our case $\zeta(s)=as$. In general, it seems that cases where $\zeta(s)$ is a contraction on $\{s\in\mathbb{C}:Re(s)\geq 0\}$ can be handled in a similar fashion. 
    
    Similarly, for the case where
    \begin{equation*}
        \tilde{Z}(s)=H(\theta(s))\tilde{Z}(\zeta(s))+\tilde{V}(s),
    \end{equation*}
where $\zeta(s)$, $\theta(s)$ should be commutative contraction mappings on $\{s\in\mathbb{C}:Re(s)\geq 0\}$; see Section \ref{shot}.

    In this direction consider the case where there is an additional dependence structure among the interarrival time and the service time of the previous customer. In particular consider the case where $A_{n+1}=\Omega_{n}S_{n}+J_{n}$, thus, $W_{n+1}=[aW_{n}+(1-\Omega_{n})S_{n}-J_{n}]^{+}$, with $J_{n}\sim \exp(\lambda_{j})$, given that $Y_{n+1}=j$, $j=1,\ldots,N$, and $\Omega_{n}$ such that $P(\Omega_{n}=u_{k})=b_{k}$, $u_{k}\in(0,1)$, $k=1,\ldots,K$, $\sum_{k=1}^{K}b_{k}=1$. Then, following the lines in Theorem \ref{th}, we obtain after some algebra:
    \begin{equation}
    Z_{j}(s)=\frac{\lambda_{j}}{\lambda_{j}-s}\sum_{i=1}^{N}p_{i,j}Z_{i}(as)\sum_{k=1}^{K}b_{k}\beta_{i}^{*}(\bar{u}_{k}s)-\frac{s}{\lambda_{j}-s}v_{j},\label{op1}
\end{equation}
where now $v_{j}:=\sum_{i=1}^{N}p_{i,j}Z_{i}(a\lambda_{j})\sum_{k=1}^{K}b_{k}\beta_{i}^{*}(\bar{u}_{k}\lambda_{j})$, $j=1,\ldots,N$, with $\bar{u}_{k}:=1-u_{k}$, $k=1,\ldots,K$. In matrix notation, \eqref{op1} is rewritten as
\begin{displaymath}
    \tilde{Z}(s)=H(\bar{u}_{1}s,\ldots,\bar{u}_{K}s)\tilde{Z}(as)+\tilde{V}(s),
\end{displaymath}
where now $H(\bar{u}_{1}s,\ldots,\bar{u}_{K}s)=L(s)P^{T}B^{*}(\bar{u}_{1}s,\ldots,\bar{u}_{K}s)$, with 
\begin{displaymath}
    B^{*}(\bar{u}_{1}s,\ldots,\bar{u}_{K}s):=diag(\sum_{k=1}^{K}b_{k}\beta_{1}^{*}(\bar{u}_{k}s),\ldots,\sum_{k=1}^{K}b_{k}\beta_{N}^{*}(\bar{u}_{k}s)).
\end{displaymath}
Using similar arguments as in the proof of Theorem \ref{th} we obtain
\begin{equation}
    \tilde{Z}(s)=s\sum_{n=0}^{\infty}a^{n}\prod_{m=0}^{n-1}H(a^{m}\bar{u}_{1}s,\ldots,a^{m}\bar{u}_{K}s)F(a^{n}s)\tilde{v}+\prod_{m=0}^{\infty}H(a^{m}\bar{u}_{1}s,\ldots,a^{m}\bar{u}_{K}s)\tilde{\pi},\,\,\tilde{Z}(0)=\tilde{\pi}.\label{siterrt}
\end{equation}
Setting $s=a\lambda_{j}$, $j=1,\ldots,N$ in \eqref{siterrt} and following the lines in Proposition \ref{prop1}, we obtain $\tilde{v}$. Note that in deriving \eqref{siterrt} we have to appropriately modify Lemma \ref{lemma1} to show the convergence of the involved infinite products and series. This is not a difficult task since $a^{m}\bar{u}_{k}\to 0$ as $m\to\infty$ for all $k=1,\ldots,K$.
\end{remark}
\subsection{The case where interarrival times are mixed-Erlang distributed}\label{xc2}
Assume now that if the Markov chain is in state $j$, the interarrival time is with probability $q_{m}$ equal to a random variable $D_{m}$, $m=1,\ldots,M$, that follows an Erlang distribution with parameter $\lambda_{j}$ and $m$ phases. More precisely, the distribution function of $A$ is given by
\begin{equation}
    G_{A,j}(x):=\sum_{m=1}^{M}q_{m} \left(1-e^{-\lambda_{j}x}\sum_{l=0}^{m-1}\frac{(\lambda_{j}x)^{l}}{l!} \right),\,x\geq 0.\label{pol}
\end{equation}
    It is well known that such a class of phase-type distributions can be used to approximate any given distribution on $[0,\infty)$ for the interarrival
times arbitrarily close; see e.g., \cite{schass}. Then, the Laplace-Stieltjes transform of the limiting distribution of the workload is given by the following theorem. We denote the
derivative of order $i$ of a function $f$ by $f^{(i)}$ with the convention that $f^{(0)}=f$.
\begin{theorem}
    Let $ G_{A,j}(.)$ be given in \eqref{pol}. Then, for $j=1,\ldots,N,$
    \begin{equation}
        \begin{array}{rl}
            Z_{j}(s)= &\sum_{i=1}^{N}p_{i,j}\sum_{m=1}^{M}q_{m}(\frac{\lambda_{j}}{\lambda_{j}-s})^{m}Z_{i}(as)\beta_{i}^{*}(s)\vspace{2mm} \\
             &+ \sum_{i=1}^{N}p_{i,j}\sum_{m=1}^{M}q_{m}\sum_{l=0}^{m-1}\sum_{k=0}^{l}\frac{(-\lambda_{j})^{l}}{l!}\binom{l}{k}\left[Z_{i}(a\lambda_{j})\right]^{(k)}\left[\beta^{*}_{i}(\lambda_{j})\right]^{(l-k)}\left(1-(\frac{\lambda_{j}}{\lambda_{j}-s})^{m-l}\right).
        \end{array}\label{bnb1}
    \end{equation}
    Equivalently, in matrix notation, the transform vector $\tilde{Z}(s)$ satisfies \begin{equation}
    \tilde{Z}(s)=\widehat{H}(s)\tilde{Z}(as)+\widehat{V}(s),
    \label{bhjn}
\end{equation}
where now $\widehat{H}(s)=\widehat{L}(s)P^{T}B^{*}(s)$, where $\widehat{L}(s):=(\sum_{m=1}^{M}q_{m}\left(\frac{\lambda_{1}}{\lambda_{1}-s}\right)^{m},\ldots,\sum_{m=1}^{M}q_{m}\left(\frac{\lambda_{N}}{\lambda_{N}-s}\right)^{m})$ and $\widehat{V}(s):=(v_{1}(s),\ldots,v_{N}(s))^{T}$ be a $N\times 1$ column vector where:
\begin{displaymath}
    v_{j}(s):=\sum_{m=1}^{M}q_{m}\widehat{e}_{m}A_{j}^{(m)}(s)\bar{v}_{j}^{(m)},\,j=1,\ldots,N.
\end{displaymath}
For $m=1,\ldots,M$, $\widehat{e}_{m}$ be a $1\times m$ row vector of ones, $A_{j}^{(m)}(s)$ be an $m\times m$ diagonal matrix whose $(t,t)$ entry equals $1-(\frac{\lambda_{j}}{\lambda_{j}-s})^{m-t+1}$, $t=1,\ldots,m$, and $\bar{v}_{j}^{(m)}:=(v_{j}^{(m)}(0),\ldots,v_{j}^{(m)}(m-1))^{T}$, where for $l=0,1,\ldots,m-1,$
\begin{displaymath}
    v_{j}^{(m)}(l):=\sum_{k=0}^{l}\frac{(-\lambda_{j})^{l}}{l!}\binom{l}{k}\sum_{i=1}^{N}p_{i,j}[Z_{i}(a\lambda_{j})]^{(k)}[\beta_{i}^{*}(\lambda_{j})]^{(l-k)},\,j=1,\ldots,N.
\end{displaymath}
\end{theorem}
\begin{proof}
    By following similar steps as in Theorem \ref{th} but for the present interarrival time we have:
    \begin{displaymath}
    \begin{array}{rl}
        Z_{j}^{n+1}(s)= &E\left(e^{-sW_{n+1}}1_{\{Y_{n+1}=j\}}\right)\vspace{2mm}\\
         =&\sum_{i=1}^{N}P(Y_{n}=i)p_{i,j} \left[E\left(\int_{0}^{aW_{n}+S_{n}}e^{-s(aW_{n}+S_{n}-y)} \sum_{m=1}^{M}q_{m}\frac{\lambda_{j}^{m}y^{m-1}e^{-\lambda_{j}y}}{(m-1)!}dy|Y_{n}=i\right)\right.\vspace{2mm}\\&\left.+E\left(\int_{aW_{n}+S_{n}}^{\infty}\sum_{m=1}^{M}q_{m}\frac{\lambda_{j}^{m}y^{m-1}e^{-\lambda_{j}y}}{(m-1)!}dy|Y_{n}=i\right)
         \right]\vspace{2mm}\\
         =&\sum_{i=1}^{N}P(Y_{n}=i)p_{i,j}\sum_{m=1}^{M}q_{m}\frac{\lambda_{j}^{m}}{(m-1)!}\left[E\left(e^{-s(aW_{n}+S_{n})} \int_{0}^{aW_{n}+S_{n}}y^{m-1}e^{-(\lambda_{j}-s)y}dy|Y_{n}=i\right)\right.\vspace{2mm}\\&\left.+E\left(\int_{aW_{n}+S_{n}}^{\infty}y^{m-1}e^{-\lambda_{j}y}dy|Y_{n}=i\right)\right]\vspace{2mm}\\
         &=\sum_{i=1}^{N}P(Y_{n}=i)p_{i,j}\sum_{m=1}^{M}q_{m}E\left(\left(\frac{\lambda_{j}}{\lambda_{j}-s}\right)^{m}e^{-s(aW_{n}+S_{n})}\right.\vspace{2mm}\\&\left.-\left(\frac{\lambda_{j}}{\lambda_{j}-s}\right)^{m}e^{-\lambda_{j}(aW_{n}+S_{n})}\sum_{l=0}^{m-1}\frac{(\lambda_{j}-s)^{l}(aW_{n}+S_{n})^{l}}{l!}+e^{-\lambda_{j}(aW_{n}+S_{n})}\sum_{l=0}^{m-1}\frac{(\lambda_{j}(aW_{n}+S_{n}))^{l}}{l!}|Y_{n}=i\right),
    \end{array}
\end{displaymath}
which results in
\begin{equation*}
        \begin{array}{rl}
            Z_{j}^{n+1}(s)= &\sum_{i=1}^{N}p_{i,j}\sum_{m=1}^{M}q_{m}(\frac{\lambda_{j}}{\lambda_{j}-s})^{m}Z_{i}^{n}(as)\beta_{i}^{*}(s)\vspace{2mm} \\
             &+ \sum_{i=1}^{N}p_{i,j}\sum_{m=1}^{M}q_{m}\sum_{l=0}^{m-1}\sum_{k=0}^{l}\frac{(-\lambda_{j})^{l}}{l!}\binom{l}{k}[Z_{i}^{n}(a\lambda_{j})]^{(k)}[\beta^{*}_{i}(\lambda_{j})]^{(l-k)}\left(1-(\frac{\lambda_{j}}{\lambda_{j}-s})^{m-l}\right).
        \end{array}\label{bn1}
    \end{equation*}
    Letting $n\to\infty$ we have that $Z_{j}(s)$ is given by \eqref{bnb1}. After simple computations we obtain the matrix representation in \eqref{bhjn}.
\end{proof}

Iterating \eqref{bhjn}, and modifying appropriately Lemma \ref{lemma1} we obtain:
\begin{equation}
    \tilde{Z}(s)=\sum_{n=0}^{\infty}\prod_{d=0}^{n-1}\widehat{H}(a^{d}s)\widehat{V}(a^{n}s)+\prod_{n=0}^{\infty}\widehat{H}(a^{n}s)\tilde{\pi},\,\tilde{Z}(0)=\tilde{\pi}.\label{soli}
\end{equation}

For $i,j=1,\ldots,N$, the $N^{2}M$ unknown constants $[Z_{i}(a\lambda_{j})]^{(m)}:=Z_{i}^{(m)}(a\lambda_{j})$, $m=0,1,\ldots,M-1$ are the unique solution to the system of linear equations resulting by  differentiating $m$ times, $m=0,1,\ldots,M-1$, and substituting $s=a\lambda_{k}$, $k=1,\ldots,N$ in \eqref{soli}. More precisely, for \textit{each} $m=0,1,\ldots,M-1$, we substitute $s=a\lambda_{k}$, $k=1,\ldots,N$ in \eqref{soli} and obtain $N^{2}$ equations that relate  
\begin{displaymath}
    \begin{array}{rl}
        \tilde{Z}^{(m)}(a\lambda_{1})=&(Z_{1}^{(m)}(a\lambda_{1}),Z_{2}^{(m)}(a\lambda_{1}),\ldots,Z_{N}^{(m)}(a\lambda_{1}))^{T},  \\
         \tilde{Z}^{(m)}(a\lambda_{2})=&(Z_{1}^{(m)}(a\lambda_{2}),Z_{2}^{(m)}(a\lambda_{2}),\ldots,Z_{N}^{(m)}(a\lambda_{2}))^{T},\\
         \vdots&\vdots\\
         \tilde{Z}^{(m)}(a\lambda_{N})=&(Z_{1}^{(m)}(a\lambda_{N}),Z_{2}^{(m)}(a\lambda_{N}),\ldots,Z_{N}^{(m)}(a\lambda_{N}))^{T}.
    \end{array}
\end{displaymath}
At the end of this procedure we will have $N^{2}M$ linear equations for the $N^{2}M$ unknown terms $Z_{i}^{(m)}(a\lambda_{j})$, $m=0,1,\ldots,M-1$, $i,j=1,\ldots,N$.
\section{Dependence based on the FGM copula}\label{fgm}
Contrary to the case in Section \ref{sec1}, in the following we assume that $A_{n+1}$, $S_{n}$ are not conditionally independent given $Y_{n}$, $Y_{n+1}$, but they are dependent based on the FGM copula. In general, copulas are functions which join or couple multivariate distribution functions to their one-dimensional marginal distribution functions \cite{nelsen}. Thus, a copula itself is a multivariate distribution function whose inputs are the
respective marginal cumulative probability distribution functions for the random variables
of interest.

A bivariate copula $C$ is a joint distribution function on $[0,1]\times[0,1]$ with standard uniform marginal  distributions. Under Sklar's theorem \cite{nelsen}, any bivariate distribution function $F$ with marginals $F_1$ and $F_2$ can   be   written   as $F(x,y)=C(F_{1}(x),F_{2}(y))$, for some copula $C$. For more details on copulas and their properties see \cite{joe,nelsen}. Modeling the dependence structure between random variables (r.v.) using copulas has become popular in actuarial science and financial risk management; e.g. \cite{denoi,albteu,cossette1} (non-exhaustive list). In queueing literature there has been some recent works where copulas are used to introduce dependency. In \cite{trapa}, the authors applied a new constructed bound copula to analyze the dependency between two
parallel service times. In \cite{lei}, the authors developed a new method based on copulas to model correlated inputs in discrete-event stochastic systems. Motivated by recent empirical evidence, the authors in \cite{wu1,wu2} used a fluid model with bivariate dependence orders and copulas to study the dependence among service times and patience times in large scale service systems; see also \cite{gu,wang,rai}. In our work we focus on a dependence structure based on the FGM copula, which is defined by
\begin{displaymath}
    C_{\theta}^{FGM}(u_{1},u_{2})=u_{1}u_{2}+\theta u_{1}u_{2}(1-u_{1})(1-u_{2}),\,(u_{1},u_{2})\in[0,1]\times[0,1],\,\theta\in[-1,1].
\end{displaymath}
The density function associated to the above expression is
\begin{displaymath}
    c_{\theta}^{FGM}(u_{1},u_{2}):=\frac{\partial^{2}}{\partial u_{1}\partial u_{2}}C_{\theta}^{FGM}(u_{1},u_{2})=1+\theta(1-2u_{1})(1-2u_{2}).
\end{displaymath}

The FGM copula allows for negative and positive dependence, and includes the independence copula ($\theta=0$), which corresponds to the case considered in Section \ref{sec1}. Our primary motivation for choosing FGM copula is due to its simplicity and its tractability due to its polynomial structure. It has a simple analytical form, which is easy to apply, thus, make it attractive to when we are dealing with the modeling of bivariate dependent data. Its simple analytic shape enables closed-form solutions to many problems in applied probability. Moreover, it is a first order approximation of the Plackett copula and of the Frank copula \cite[p. 100, and p. 106, respectively]{nelsen}. For applications of FGM copula in risk theory, health insurance plans, financial risk management, stochastic frontiers and a stereological context, see \cite{cossette1}, \cite{mai}, \cite{barges}, \cite{kim} and the references therein. Often, it is natural to describe the FGM copula as a perturbation of the product copula (e.g. \cite{durante}), inducing moderate dependence between marginals. The class of FGM copulas is a popular choice when working in two dimensions due to its simple shape and the exact calculus of polynomial functions.

In general, the bivariate distribution, say $F_{X,Y}$, of the bivariate random vector $(X,Y)$ with continuous marginals, say $F_{X}$, $F_{Y}$, and with a dependence structure based on the FGM copula is defined by
\begin{displaymath}
\begin{array}{rl}
   F_{X,Y}(x,y)=&C_{\theta}^{FGM}(F_{X}(x),F_{Y}(y))\vspace{2mm}\\
     =& F_{X}(x)F_{Y}(y)+\theta F_{X}(x)F_{Y}(y)\left(1-F_{X}(x)\right)\left(1-F_{Y}(y)\right),\,x,y\in\mathbb{R}^{+},\,\theta\in[-1,1].
\end{array}
\end{displaymath}
 With this, the bivariate density of $(X,Y)$ is given by,
\begin{displaymath}
\begin{array}{rl}
   f_{X,Y}(x,y)=&c_{\theta}^{FGM}(F_{X}(x),F_{Y}(y))f_{X}(x)f_{Y}(y)\vspace{2mm}\\
     =& f_{X}(x)f_{Y}(y)+\theta g(x)f_{Y}(y)\left(2\bar{F}_{Y}(y)-1\right),\,x,y\in\mathbb{R}^{+},\,\theta\in[-1,1],
\end{array}
\end{displaymath}
where $g(x)=f_{X}(x)(1-2F_{X}(x))$ with Laplace transform $g^{*}(s)=\int_{0}^{\infty}e^{-sx}g(x)dx$.

In the following, we focus on the Markov-dependent reflected autoregressive process, and assume that $\{(S_{n},A_{n+1}),n\geq 1\}$ form a sequence of i.i.d. random vectors distributed as the canonical r.v. $(S,A)$, in which the components are dependent based on the FGM copula. To our best knowledge, there is no other work in the related literature that focus on (either scalar or Markov-dependent) reflected autoregressive processes with dependencies based on copulas. Thus, we are dealing with a recursion of the form $W_{n+1}=[aW_{n}+S_{n}-A_{n+1}]^{+}$, where now for $n\geq 0$, $x,y\geq 0$, $i,j=1,\ldots,N$:
\begin{equation}
    \begin{array}{l}
         P(A_{n+1}\leq x,S_{n}\leq y,Y_{n+1}=j|Y_{n}=i,A_{2},\ldots,A_{n},S_{1},\ldots,S_{n-1},Y_{1},\ldots,Y_{n-1}) \vspace{2mm}\\
        =  P(A_{n+1}\leq x,S_{n}\leq y,Y_{n+1}=j|Y_{n}=i)=p_{i,j}F_{S,A|i,j}(y,x),
    \end{array}\label{mp1m}
\end{equation}
where, $F_{S,A|i,j}(y,x)$ is the bivariate distribution function of $(S_{n},A_{n+1})$ given $Y_{n}$, $Y_{n+1}$ with marginals $F_{S,i}(y)$, $F_{A,j}(x)$ defined as $F_{S,A|i,j}(y,x)=C_{\theta}^{FGM}(F_{S,i}(y),F_{A,j}(x))$ for $(y,x)\in \mathbb{R}^{+}\times \mathbb{R}^{+}$. The bivariate density of $(S,A)$ is given by 
\begin{displaymath}
\begin{array}{rl}
     f_{S,A|i,j}(y,x)=&c_{\theta}^{FGM}(F_{S,i}(y),F_{A,j}(x))f_{S,i}(y)f_{A,j}(x)= f_{S,i}(y)f_{A,j}(x)+\theta g_{i}(y)(2\bar{F}_{A,j}(x)-1)f_{A,j}(x),
\end{array}
\end{displaymath}
with $g_{i}(y):=f_{S,i}(y)(1-2F_{S,i}(y))$ with Laplace transform $g_{i}^{*}(s)=\int_{0}^{\infty}e^{-sy}g_{i}(y)dy$, and $\bar{F}_{A,j}(x):=1-F_{A,j}(x)$. In our case,
\begin{equation}
    \begin{array}{c}
         f_{S,A|i,j}(y,x)=f_{S,i}(y)\lambda_{j}e^{-\lambda_{j}x}+\theta g_{i}(y)\left[2\lambda_{j}e^{-2\lambda_{j}x}-\lambda_{j}e^{-\lambda_{j}x}\right].
    \end{array}\label{biv1}   
\end{equation}

Let now $Z_{i}^{n}(s;\theta):=E(e^{-sW_{n}}1_{\{Y_{n}=i\}};\theta)$, $Re(s)\geq 0$, $i=1,\ldots,N$, $n\geq 0$, and assuming the limit exists, define $Z_{i}(s;\theta)=\lim_{n\to\infty}Z_{i}^{n}(s;\theta)$, $i=1,\ldots,N$. By using similar argument as above, we obtain the following result.
\begin{theorem}\label{th2}
    The transforms $Z_{j}(s;\theta)$, $j=1,\ldots,N$, satisfy the following equation:
    \begin{equation}
        \begin{array}{rl}
            Z_{j}(s;\theta)= &\frac{\lambda_{j}}{\lambda_{j}-s}\sum_{i=1}^{N}p_{i,j}\left(\beta_{i}^{*}(s)-\frac{\theta s}{2\lambda_{j}-s}g_{i}(s)\right)Z_{i}(as;\theta) \vspace{2mm} \\
             &-s\sum_{i=1}^{N}p_{i,j}\left[\frac{\theta}{2\lambda_{j}-s}g_{i}(2\lambda_{j})Z_{i}(2a\lambda_{j};\theta)+\frac{\beta_{i}^{*}(\lambda_{j})-\theta g_{i}(\lambda_{j})}{\lambda_{j}-s}Z_{i}(a\lambda_{j};\theta)\right]. 
        \end{array}\label{fgma}
    \end{equation}
    In matrix notation, \eqref{fgma} is rewritten as 
    \begin{equation}
        \Tilde{Z}(s;\theta)=U(s;\theta)\Tilde{Z}(as;\theta)+\Tilde{K}(s;\theta),\label{gh}
    \end{equation}
    where now, 
    \begin{displaymath}
        \begin{array}{rl}
            U(s;\theta):= &L_{1}(s)[P^{T}B^{*}(s)+\theta(I-L_{2}(s))P^{T}G^{*}(s)],\vspace{2mm}\\
            \Tilde{K}(s;\theta)=& (I-L_{1}(s))\tilde{v}^{(1)}+\theta (I-L_{2}(s))\tilde{v}^{(2)},
        \end{array}
    \end{displaymath}
    with  $G^{*}(s):=diag(g_{1}^{*}(s),\ldots,g_{N}^{*}(s))$, $L_{k}(s)=diag(\frac{k\lambda_{1}}{k\lambda_{1}-s},\ldots,\frac{k\lambda_{N}}{k\lambda_{N}-s})$, $\tilde{v}^{(k)}:=(v_{1}^{(k)},\ldots,v_{N}^{(k)})$, $k=1,2,$ where for $j=1,\ldots,N$,
    \begin{displaymath}
   v_{j}^{(1)}:=\sum_{i=1}^{N}p_{i,j} \left(\beta_{i}^{*}(\lambda_{j})-\theta g_{i}^{*}(\lambda_{j}) \right)  Z_{i}(a\lambda_{j};\theta),\,\,\,v_{j}^{(2)}:=\theta\sum_{i=1}^{N}p_{i,j} g_{i}^{*}(2\lambda_{j})Z_{i}(2a\lambda_{j};\theta).
    \end{displaymath} 
\end{theorem}
\begin{proof}
     \begin{displaymath}
    \begin{array}{rl}
        Z_{j}^{n+1}(s;\theta)= &E\left(e^{-sW_{n+1}}1_{\{Y_{n+1}=j\}};\theta\right)\vspace{2mm}\\
         =&\sum_{i=1}^{N}p_{i,j}E\left(\int_{x=0}^{\infty}\int_{y=0}^{a W_{n}+x}e^{-s(aW_{n}+x-y)}\left[f_{S,i}(x)\lambda_{j} e^{-\lambda_{j}y}+\theta g_{i}(x)\left(2\lambda_{j} e^{-2\lambda_{j}y}-\lambda_{j} e^{-\lambda_{j}y}\right)\right]dydx\right.\vspace{2mm}\\
         &\left.+\int_{x=0}^{\infty}\int_{y=a W_{n}+x}^{\infty}\left[f_{S,i}(x)\lambda_{j} e^{-\lambda_{j}y}+\theta g_{i}(x)\left(2\lambda_{j} e^{-2\lambda_{j}y}-\lambda_{j} e^{-\lambda_{j}y}\right)\right]dydx|Y_{n}=i\right)\vspace{2mm}\\
         =&\sum_{i=1}^{N}p_{i,j}\left(Z_{i}^{n}(as;\theta)\left[\frac{\lambda_{j}}{\lambda_{j}-s}\beta_{i}^{*}(s)-\frac{\theta\lambda_{j}sg_{i}^{*}(s)}{(2\lambda_{j}-s)(\lambda_{j}-s)}\right]\right.\vspace{2mm} \\
         &\left.-\frac{s\theta g_{i}^{*}(s)}{2\lambda_{j}-s}Z_{i}^{n}(2a\lambda_{j};\theta)-\frac{s}{\lambda_{j}-s}\left[\beta_{i}^{*}(\lambda_{j})-\theta g_{i}^{*}(\lambda_{j})\right]Z_{i}^{n}(a\lambda_{j};\theta)\right).
    \end{array}
\end{displaymath}
Letting $n\to\infty$ yields \eqref{fgma}. Simple algebraic calculations lead to \eqref{gh}. 
\end{proof}

Iterating \eqref{gh}, and modifying appropriately Lemma \ref{lemma1} we obtain:
\begin{equation}
    \tilde{Z}(s;\theta)=\sum_{n=0}^{\infty}\prod_{d=0}^{n-1}U(a^{d}s;\theta)\Tilde{K}(a^{n}s;\theta)+\prod_{n=0}^{\infty}U(a^{n}s;\theta)\tilde{\pi},\,\tilde{Z}(0;\theta)=\tilde{\pi}.\label{siterr1}
\end{equation}

Note that we still need to obtain $Z_{i}(a\lambda_{j};\theta)$, $Z_{i}(2a\lambda_{j};\theta)$, $i,j=1,\ldots,N$, or equivalently, to obtain the column vectors $\tilde{v}^{(k)}$, $k=1,2$. The following proposition provides information about how we can obtain these vectors.
\begin{proposition}\label{prop2}
    The column vectors $\tilde{v}^{(k)}$, $k=1,2$ are derived as the unique solution of the following system of $2N$ equations:
    \begin{equation}
        \begin{array}{rl}
          v_{j}^{(1)}=   &e_{j}P^{T}(B^{*}(\lambda_{j})-\theta G^{*}(\lambda_{j}))\left[\sum_{n=0}^{\infty}\prod_{d=0}^{n-1}U(a^{d+1}\lambda_{j};\theta)\Tilde{K}(a^{n+1}\lambda_{j};\theta)+\prod_{n=0}^{\infty}U(a^{n+1}\lambda_{j};\theta)\tilde{\pi}\right],\,j=1,\ldots,N,\vspace{2mm}  \\
             v_{j}^{(2)}=   &\theta e_{j}P^{T}G^{*}(2\lambda_{j})\left[\sum_{n=0}^{\infty}\prod_{d=0}^{n-1}U(2a^{d+1}\lambda_{j};\theta)\Tilde{K}(2a^{n+1}\lambda_{j};\theta)+\prod_{n=0}^{\infty}U(2a^{n+1}\lambda_{j};\theta)\tilde{\pi}\right],\,j=1,\ldots,N.
        \end{array}\label{lpo1}
    \end{equation}
\end{proposition}
\begin{proof}
    The proof is similar to the one in Proposition \ref{prop1} and further details are omitted.
\end{proof}
\begin{remark}
    Note that for $\theta=0$, i.e., the independent copula, Theorem \ref{th2} is reduced to Theorem \ref{th} in Section \ref{sec1}. Indeed, setting $\theta=0$, $F_{S,A|i,j}(y,x)=F_{S,i}(y)F_{A,j}(x)$, i.e., $S_{n}$, $A_{n+1}$ are conditionally independent, given $Y_{n}$ and $Y_{n+1}$, and thus, by substituting $\theta=0$ in \eqref{biv1}, $f_{S,A|i,j}(y,x)=f_{S,i}(y)\lambda_{j}e^{-\lambda_{j}x}$, $i,j\in E = \{1, 2, ..., N\}$. Under such a setting, straightforward computations imply that \eqref{fgma} becomes identical to \eqref{op}, or equivalently \eqref{gh} becomes identical to \eqref{bhjn}. Note also that when $\theta=0$, we only need to derive the vector $\tilde{v}^{(1)}$ (the vector $\tilde{v}^{(2)}$ is no longer necessary).
\end{remark}
\begin{remark}
    The analysis can be extended to the case where given the Markov chain is in state $j$, $A_{n}$ to follow a mixed Erlang distribution at a cost of more complicated expressions; see subsection \ref{xc2}.\end{remark}
    \begin{remark}
    Note that in this section (except the generalization due to the Markov-dependent framework), we further generalize the work in \cite{box1} (and the one in \cite{dimi}) by considering dependencies based on FGM copula. In particular, by setting $N=1$ and $\theta=0$ (i.e., the non-modulated scenario and the independent copula), our model reduces to the one in \cite{box1}. When $N=1$, $\theta\in[-1,1]$, we generalize the work in \cite{box1} by considering the dependence based on FGM copula among $\{S_{n}\}_{n\in\mathbb{N}_{0}}$, $\{A_{n}\}_{n\in\mathbb{N}}$. To our best knowledge, this is the first work that deals with the stationary analysis of Markov-dependent reflected autoregressive processes in the presence of copulas. 
\end{remark}
\section{Dependence based on a class of multivariate matrix-exponential distributions}\label{bila}
We now assume the case where given that $Y_{n}=i$, $Y_{n+1}=j$, the i.i.d. random vector $(S_{n},A_{n+1})$ has a bivariate matrix exponential distribution, thus, its joint Laplace-Stieltjes transform $E\left(e^{-s_{1}S_{n}-s_{2}A_{n+1}}|Y_{n}=i,Y_{n+1}=j\right)$ is a rational function in $s_{1}$, $s_{2}$, i.e., it is written as $\frac{G_{i,j}(s_{1},s_{2})}{D_{i,j}(s_{1},s_{2})}$, where $G_{i,j}(s_{1},s_{2})$, $D_{i,j}(s_{1},s_{2})$ are polynomial functions in $s_{1}$, $s_{2}$. Thus the LST of the difference $X_{n}:=S_{n}-A_{n+1}$ is also a rational function, say $E\left(e^{-sX_{n}}|Y_{n}=i,Y_{n+1}=j\right):=\frac{f_{i,j}(s)}{g_{i,j}(s)}$, $i,j\in E$. Then, letting $U_{n}:=[aW_{n}+S_{n}-A_{n+1}]^{-}$, and using the identity $1+e^{x}=e^{[x]^{+}}+e^{[x]^{-}}$, where $[x]^{+}:=max(0,x)$, $[x]^{-}:=min(0,x)$, we obtain  
\begin{displaymath}
    \begin{array}{rl}
         Z_{j}^{n+1}(s)=& E\left(e^{-sW_{n+1}}1_{\{Y_{n+1}=j\}}\right)=\sum_{i=1}^{N}P(Y_{n}=i)E\left(e^{-sW_{n+1}}1_{\{Y_{n+1}=j\}}|Y_{n}=i\right)\vspace{2mm}\\=&\sum_{i=1}^{N}P(Y_{n}=i)E\left(e^{-s[aW_{n}+S_{n}-A_{n+1}]^{+}}1_{\{Y_{n+1}=j\}}|Y_{n}=i\right)  \vspace{2mm} \\
         =& \sum_{i=1}^{N}P(Y_{n}=i)E\left(\left(e^{-s(aW_{n}+S_{n}-A_{n+1})}+1-e^{-sU_{n}}\right)1_{\{Y_{n+1}=j\}}|Y_{n}=i\right) \vspace{2mm} \\
         =& \sum_{i=1}^{N}P(Y_{n}=i)p_{i,j}E\left(e^{-s(aW_{n}+S_{n}-A_{n+1})}+1-e^{-sU_{n}}|Y_{n+1}=j,Y_{n}=i\right)\vspace{2mm} \\
         =&\sum_{i=1}^{N}P(Y_{n}=i)p_{i,j}E\left(e^{-saW_{n}}|Y_{n}=i\right)E\left(e^{-s(S_{n}-A_{n+1})}|Y_{n+1}=j,Y_{n}=i\right)\vspace{2mm} \\&+1-\sum_{i=1}^{N}P(Y_{n}=i)p_{i,j}E\left(e^{-sU_{n}}|Y_{n+1}=j,Y_{n}=i\right)\vspace{2mm} \\
         =&\sum_{i=1}^{N}p_{i,j}Z_{i}^{n}(as)\frac{f_{i,j}(s)}{g_{i,j}(s)}+1-U_{j,n}^{-}(s),
    \end{array}
\end{displaymath}
where $U_{j,n}^{-}(s):=\sum_{i=1}^{N}P(Y_{n}=i)p_{i,j}E\left(e^{-sU_{n}}|Y_{n+1}=j,Y_{n}=i\right)$. Letting $n\to\infty$ so that $Z_{j}^{n}(s)\to Z_{j}(s)$, $U_{j,n}^{-}(s)\to U_{j}(s)$ we arrive at,
\begin{equation}
    Z_{j}(s)=\sum_{i=1}^{N}p_{i,j}Z_{i}(as)\frac{f_{i,j}(s)}{g_{i,j}(s)}+1-U_{j}^{-}(s).\label{eq1}
\end{equation}

For $i,j=1,\ldots,N$, let $m_{+}(i,j)$ be the number of zeros of $g_{i,j}(s)$ in $Re(s)\geq 0$, and let $g_{i,j}^{+}(s):=\prod_{l=1}^{m_{+}(i,j)}(s-s_{l,j}^{+})$, i.e., $s_{l,j}^{+}$, $l=1,\ldots,m_{+}(i,j)$, $j=1,\ldots,N$ are the zeros of $g_{i,j}(s)$ in $Re(s)\geq 0$; and $g^{-}_{i,j}(s)=g_{i,j}(s)/g_{i,j}^{+}(s)$. Thus, we can write for $j=1,\ldots,N,$
\begin{equation}
    \begin{array}{c}
         Z_{j}(s)\prod_{i=1}^{N}g_{i,j}^{+}(s)-\sum_{i=1}^{N}p_{i,j}f_{i,j}(s)\frac{\prod_{k\neq i}g_{k,j}(s)}{\prod_{t=1}^{N}g_{t,j}^{-}(s)}Z_{i}(as)=\prod_{i=1}^{N}g_{i,j}^{+}(s)[1-U_{j}^{-}(s)]. 
    \end{array}\label{fty}
\end{equation}
Then, we have the following:
\begin{itemize}
    \item The LHS of \eqref{fty} is analytic in $Re(s) > 0$ and continuous in $Re(s)\geq 0$.
    \item The RHS of \eqref{fty} is analytic in $Re(s) < 0$ and continuous in $Re(s)\leq 0$.
    \item $Z_{j}(s)$ is for $Re(s)\geq 0$ bounded by 1, and hence, the LHS of \eqref{fty} behaves at most as a polynomial of degree at most $\sum_{i=1}^{N}m_{+}(i,j)$ in $s$ for large $s$, with $Re(s) > 0$.
    \item $U_{j}^{-}(s)$ is for $Re(s) \leq  0$ bounded by 1, and hence, the RHS of \eqref{fty} behaves at most
as a polynomial of degree at most $\sum_{i=1}^{N}m_{+}(i,j)$ in $s$ for large $s$, with $Re(s) < 0$.
\end{itemize}
Thus, Liouville’s theorem \cite[Theorem 10.52]{tit} implies that both sides in \eqref{fty} are equal
to the same polynomial in s, in their respective half-planes. Therefore, for $Re(s)\geq 0$, and $j=1,\ldots,N,$
\begin{equation}
Z_{j}(s)\prod_{i=1}^{N}g_{i,j}^{+}(s)-\sum_{i=1}^{N}p_{i,j}f_{i,j}(s)\frac{\prod_{k\neq i}g_{k,j}(s)}{\prod_{t=1}^{N}g_{t,j}^{-}(s)}Z_{i}(as)= \sum_{m=0}^{\sum_{i=1}^{N}m_{+}(i,j)}c_{m}^{(j)}s^{m}.\label{xz}
\end{equation}
Setting $s=0$ in \eqref{xz}, and having in mind that $f_{i,j}(0)=g_{i,j}(0)$ yields $c_{0}^{(j)}=0$, $j=1,\ldots,N$. In matrix notation, \eqref{xz} can be rewritten as
\begin{equation}
    \Tilde{Z}(s)=H(s)\Tilde{Z}(as)+\Tilde{V}(s),\label{xz1}
\end{equation}
where,
\begin{displaymath}\begin{array}{rl}
    H(s):=&\begin{pmatrix}
        p_{1,1}\frac{f_{1,1}(s)}{g_{1,1}(s)}&p_{2,1}\frac{f_{2,1}(s)}{g_{2,1}(s)}&\ldots&p_{N,1}\frac{f_{N,1}(s)}{g_{N,1}(s)}\\
        p_{1,2}\frac{f_{1,2}(s)}{g_{1,2}(s)}&p_{2,2}\frac{f_{2,2}(s)}{g_{2,2}(s)}&\ldots&p_{N,2}\frac{f_{N,2}(s)}{g_{N,2}(s)}\\
        \vdots&\vdots&\ddots&\vdots\\
        p_{1,n}\frac{f_{1,N}(s)}{g_{1,N}(s)}&p_{2,N}\frac{f_{2,N}(s)}{g_{2,N}(s)}&\ldots&p_{N,N}\frac{f_{N,N}(s)}{g_{N,N}(s)}
    \end{pmatrix},\vspace{2mm}\\
    \Tilde{V}(s):=&(\sum_{m=1}^{\sum_{i=1}^{N}m_{+}(i,1)}c_{m}^{(1)}s^{m},\ldots,\sum_{m=1}^{\sum_{i=1}^{N}m_{+}(i,N)}c_{m}^{(N)}s^{m})^{T}.
\end{array}
\end{displaymath}

Note that $H(a^{m}s)\to P^{T}$ as $m\to\infty$. Thus, we can apply Lemma \ref{lemma1} to study the convergence of the iterating procedure. Therefore, by iterating \eqref{xz1}, and applying appropriately Lemma \ref{lemma1} we obtain:
\begin{equation}
    \tilde{Z}(s)=\sum_{n=0}^{\infty}\prod_{d=0}^{n-1}H(a^{d}s)\Tilde{V}(a^{n}s)+\prod_{n=0}^{\infty}H(a^{n}s)\tilde{\pi},\label{xz2}
\end{equation}
with $\tilde{Z}(0)=\tilde{\pi}$. The remaining $\sum_{i=1}^{N}\sum_{j=1}^{N}m_{+}(i,j)$ unknown constants are obtained by performing the following
three steps:
\begin{enumerate}
    \item Substitute $s=s_{l,j}^{+}$, $l=1,\ldots,m_{+}(i,j)$, $i,j=1,\ldots,N$ in \eqref{fty} to obtain the following set of equations:
\begin{equation}
    -\sum_{i=1}^{N}p_{i,j}f_{i,j}(s_{l,j}^{+})\frac{\prod_{k\neq i}g_{k,j}(s_{l,j}^{+})}{\prod_{t=1}^{N}g_{t,j}^{-}(s_{l,j}^{+})}Z_{i}(as_{l,j}^{+})= \sum_{m=1}^{\sum_{i=1}^{N}m_{+}(i,j)}c_{m}^{(j)}(s_{l,j}^{+})^{m},\label{b1}
\end{equation}
thus, linearly expressing $Z_{i}(as_{l,j}^{+})$, $i,j=1,\ldots,N$, $l=1,\ldots,m_{+}(i,j)$ with the unknown constants $c_{m}^{(j)}$, $m=1,2,\ldots,\sum_{i=1}^{N}m_{+}(i,j)$, $j=1,\ldots,N$.
\item Substitute $s=as_{l,j}^{+}$, $i,j=1,\ldots,N$, $l=1,\ldots,m_{+}(i,j),$ in \eqref{xz2} to obtain:
\begin{equation}
    \tilde{Z}(as_{l,j}^{+})=\sum_{n=0}^{\infty}\prod_{d=0}^{n-1}H(a^{d+1}s_{l,j}^{+})\Tilde{V}(a^{n+1}s_{l,j}^{+})+\prod_{n=0}^{\infty}H(a^{n+1}s_{l,j}^{+})\tilde{\pi},\label{xnz2}
\end{equation}
thus, linearly expressing $Z_{i}(as_{l,j}^{+})$, $i,j=1,\ldots,N$, $l=1,\ldots,m_{+}(i,j)$ with the unknown constants $c_{m}^{(j)}$, $m=1,2,\ldots,\sum_{i=1}^{N}m_{+}(i,j)$, $j=1,\ldots,N$, in a different way.
\item Eliminate all $Z_{i}(as_{l,j}^{+})$, $i,j=1,\ldots,N$, $l=1,\ldots,m_{+}(i,j)$ from the latter $\sum_{i=1}^{N}\sum_{j=1}^{N}m_{+}(i,j)$ equations using the former $\sum_{i=1}^{N}\sum_{j=1}^{N}m_{+}(i,j)$ equations, and
then solve the resulting set of $\sum_{i=1}^{N}\sum_{j=1}^{N}m_{+}(i,j)$ linear equations in $c_{m}^{(j)}$, $m=1,2,\ldots,\sum_{i=1}^{N}m_{+}(i,j)$, $j=1,\ldots,N$.
\end{enumerate}
Thus, $\Tilde{V}(s)$ is considered known and the transform vector $\tilde{Z}(s)$ is given by \eqref{xz2}.

\section{The Markov-modulated reflected autoregressive process}\label{mm}
This section is devoted to the analysis of the transient behaviour of a Markov-modulated reflected autoregressive process, i.e., the time-dependent analysis of a generalized version of the model in Section \ref{sec1} with a more complicated dependence structure; see also Remark \ref{remm}. In queueing terms, consider an M/G/1-type queue with Markov-modulated arrivals and services. Let $\{X(t);t\geq 0\}$ be the background process that dictates the arrivals and services. $\{X(t);t\geq 0\}$ is a Markov chain on $E=\{1,2,\ldots,N\}$ with infinitesimal generator $Q=(q_{i,j})_{i,j\in E}$, and denote its stationary distribution by $\widehat{\pi}=(\pi_{1},\ldots,\pi_{N})$, i.e., $\widehat{\pi}Q=0$, and $\widehat{\pi} \mathbf{1}=1$, where $\mathbf{1}$ is the $N\times 1$ column vector with all components equal to 1.

Customers arrive at time epochs $T_{1},T_{2},\ldots$, $T_{1}=0$, and service times are denoted by $S_{1},S_{2},\ldots$. If $X(t)=i$, arrivals occur according to a Poisson process with rate $\lambda_{i}>0$ and an arriving customer has a service time $B{i}$, with cdf $B_{i}(.)$. pdf $b_{i}(.)$, LST $\beta_{i}^{*}(.)$ and $\bar{b}_{i}=-\beta_{i}^{*\prime}(0)$, $i=1,\ldots,N$, and $B^{*}(s):=diag(\beta_{1}^{*}(s),\ldots,\beta_{N}^{*}(s))$. We assume that given the state of the background process $\{X(t);t\geq 0\}$, $S_{1},S_{2},\ldots$ are independent and independent of the arrival process. Let $A_{n}=T_{n}-T_{n-1}$, $n=2,3,\ldots$, and $Y_{n}=X_{T_{n}}$, $n=1,2,\ldots$. We are interested in the workload $W_{n}$ in the system just before the $n$th customer arrival, and we assume that such an arrival makes obsolete a fixed fraction $1-a$ of the already present work, irrespectively of the state of $X(t)$. This model can be considered as the "autoregressive" analogue of the work in \cite{regte}. Under such a scenario, for $n\geq 0$, $x,y\geq 0$, $i,j=1,\ldots,N$:
\begin{equation*}
    \begin{array}{l}
         P(A_{n+1}\leq x,S_{n}\leq y,Y_{n+1}=j|Y_{n}=i,A_{2},\ldots,A_{n},S_{1},\ldots,S_{n-1},Y_{1},\ldots,Y_{n-1}) \vspace{2mm}\\
        =  P(A_{n+1}\leq x,S_{n}\leq y,Y_{n+1}=j|Y_{n}=i),
    \end{array}
\end{equation*}
and the model is fully specified by the functions
\begin{equation}
    G_{i,j}(s,\eta):=E\left(e^{-s A_{n+1}-\eta S_{n}}1_{\{Y_{n+1}=j\}}|Y_{n}=i\right)=A_{i,j}(s)\beta_{i}^{*}(\eta),\label{fun}
    \end{equation}
    where,
    \begin{displaymath}
    A_{i,j}(s):=E\left(e^{-s A_{n}}1_{\{Y_{n}=j\}}|Y_{n-1}=i\right),\,i,j\in E,
\end{displaymath}
and $G(s,\eta)$, $A(s)$ the $N\times N$ matrix with elements $G_{i,j}(s,\eta)$, $A_{i,j}(s)$, $i,j\in E$, $Re(s)\geq 0$, $Re(\eta)\geq 0$, respectively.
\begin{remark}\label{remm}
    Note that the model analysed in Section \ref{sec1} corresponds to the case where $A_{i,j}(s)=p_{i,j}u_{j}^{*}(s)$, where in subsection \ref{xc}, $u_{j}^{*}(s):=\frac{\lambda_{j}}{\lambda_{j}+s}$, while in subsection \ref{xc2}, $u_{j}^{*}(s):=\sum_{m=1}^{M}q_{m}\left(\frac{\lambda_{j}}{\lambda_{j}+s}\right)^{m}$.
\end{remark}
%\subsection{Transient distribution}

Assume that $W_{1}=w$ and let for $Re(s)\geq 0$, $Re(\eta)\geq 0$, $|r|<1$,
\begin{displaymath}
    Z_{j}^{w}(r,s,\eta):=\sum_{n=1}^{\infty}r^{n}E\left(e^{-s W_{n}-\eta T_{n}}1_{\{Y_{n}=j\}}|W_{1}=w\right),\,j=1,\ldots,N,
\end{displaymath}
and $\tilde{Z}^{w}(r,s,\eta)=(Z_{1}^{w}(r,s,\eta),\ldots,Z_{N}^{w}(r,s,\eta))^{T}$. Following the lines in \cite[Lemma 2.1]{regte}, we have that $A(s)=M^{-1}(s)\Lambda$, where $\Lambda:=diag(\lambda_{1},\ldots,\lambda_{N})$, and $M(s)=s I+\Lambda-Q$. Let also
\begin{displaymath}
    V_{j}^{w}(r,s,\eta):=\sum_{n=1}^{\infty}r^{n+1}E\left((1-e^{-s [a W_{n}+S_{n}-A_{n+1}]^{-}})e^{-\eta T_{n+1}}1_{\{Y_{n+1}=j\}}|W_{1}=w\right),\,j=1,\ldots,N,
\end{displaymath}
with $\tilde{V}^{w}(r,s,\eta):=(V_{1}^{w}(r,s,\eta),\ldots,V_{N}^{w}(r,s,\eta))^{T}$, and let $p_{j}:=P(X_{0}=j)$, $j=1,\ldots,N$ with $\widehat{p}:=(p_{1},\ldots,p_{N})^{T}$.
\begin{theorem}\label{thr}
    For $Re(s)=0$, $Re(\eta)\geq 0$, $|r|<1$,
    \begin{equation}
        Z^{w}_{j}(r,s,\eta)-rp_{j}e^{-s w}=r\sum_{i=1}^{N}Z^{w}_{i}(r,as,\eta)\beta_{i}^{*}(s)A_{i,j}(\eta-s)+V_{j}^{w}(r,s,\eta),\,j=1,2,\ldots,N,
        \label{eqx1}
    \end{equation}
    or equivalently, in matrix notation,
    \begin{equation}
        \tilde{Z}^{w}(r,s,\eta)-r\Lambda(M^{T}(\eta-s))^{-1}B^{*}(s)\tilde{Z}^{w}(r,as,\eta)=re^{-s w}\widehat{p}+\tilde{V}^{w}(r,s,\eta).\label{eq2}
    \end{equation}
\end{theorem}
\begin{proof}
    Using the identity $e^{-s [x]^{+}}+e^{-s [x]^{-}}=e^{-s x}+1$, we have for $Re(s)=0$, $Re(\eta)\geq 0$, $|r|<1$,
    \begin{displaymath}
        \begin{array}{l}
            E\left(e^{-s W_{n+1}-\eta T_{n+1}}1_{\{Y_{n+1}=j\}}|W_{1}=w\right)= E\left(e^{-s [a W_{n}+S_{n}-A_{n+1}]^{+}-\eta T_{n+1}}1_{\{Y_{n+1}=j\}}|W_{1}=w\right)\vspace{2mm}  \\
            =  E\left((e^{-s[a W_{n}+S_{n}-A_{n+1}]}+1-e^{-s [a W_{n}+S_{n}-A_{n+1}]^{-}})e^{-\eta T_{n+1}}1_{\{Y_{n+1}=j\}}|W_{1}=w\right)\vspace{2mm}\\
            =  E\left(e^{-s [a W_{n}+S_{n}-A_{n+1}]-\eta T_{n+1}}1_{\{Y_{n+1}=j\}}|W_{1}=w\right)
            +E\left((1-e^{-s [a W_{n}+S_{n}-A_{n+1}]^{-}})e^{-\eta T_{n+1}}1_{\{Y_{n+1}=j\}}|W_{1}=w\right).
        \end{array}
    \end{displaymath}
    Note that
    \begin{displaymath}
        \begin{array}{l}
            E\left(e^{-s (a W_{n}+S_{n}-A_{n+1})-\eta T_{n+1}}1_{\{Y_{n+1}=j\}}|W_{1}=w\right) =E\left(e^{-s (a W_{n}+S_{n}-A_{n+1})-\eta (A_{n+1}+T_{n})}1_{\{Y_{n+1}=j\}}|W_{1}=w\right)  \vspace{2mm}\\
            = E \left(e^{-s a W_{n}-\eta T_{n}}e^{-s S_{n}-(\eta-s)A_{n+1}}1_{\{Y_{n+1}=j\}}|W_{1}=w\right)
        \vspace{2mm}\\=\sum_{i=1}^{N}E\left(e^{-s a W_{n}-\eta T_{n}}1_{\{Y_{n+1}=j\}}|W_{1}=w\right)\beta_{i}^{*}(s)A_{i,j}(\eta-s).
        \end{array}
    \end{displaymath}
    Thus, for $Re(s)=0$, $Re(\eta)\geq 0$, $|r|<1$,
    \begin{equation}
        \begin{array}{rl}
         E\left(e^{-s W_{n+1}-\eta T_{n+1}}1_{\{Y_{n+1}=j\}}|W_{1}=w\right)=    & \sum_{i=1}^{N}E\left(e^{-s a W_{n}-\eta T_{n}}1_{\{Y_{n}=i\}}|W_{1}=w\right)\beta_{i}^{*}(s)A_{i,j}(\eta-s)\vspace{2mm} \\
         &+E\left((1-e^{-s [a W_{n}+S_{n}-A_{n+1}]^{-}})e^{-\eta T_{n+1}}1_{\{Y_{n+1}=j\}}|W_{1}=w\right).
        \end{array}\label{rty}
    \end{equation}
    Multiplying \eqref{rty} by $r^{n+1}$ and taking the sum of $n=1$ to infinity gives:
    \begin{displaymath}
        Z^{w}_{j}(r,s,\eta)-r E\left(e^{-s W_{1}-\eta T_{1}}1_{\{Y_{1}=j\}}|W_{1}=w\right)=r\sum_{i=1}^{N}Z^{w}_{i}(r,as,\eta)\beta_{i}^{*}(s)A_{i,j}(\eta-s)+V_{j}^{w}(r,s,\eta).
    \end{displaymath}
    Note that $T_{1}=0$ and,
    \begin{displaymath}
        E\left(e^{-s W_{1}-\eta T_{1}}1_{\{Y_{1}=j\}}|W_{1}=w\right)=E\left(e^{-s w}1_{\{X_{0}=j\}}\right)=e^{-s w}P(X_{0}=j)=e^{-s w}p_{j}.
    \end{displaymath}
   Substituting back in \eqref{rty} yields the system of Wiener-Hopf equations \eqref{eqx1}, which in matrix notation, is equivalent to \eqref{eq2}. 
\end{proof}
\begin{remark}\label{rem2}
Note that \eqref{eq2} is the matrix analogue of equation (2.1) in \cite{box1}, i.e., it is of the form:
    \begin{equation*}
        \mathbf{f}(r,s,\eta)=\mathbf{g}(r,s,\eta)\mathbf{f}(r,a s,\eta)+\mathbf{K}(r,s,\eta),
    \end{equation*}
    with $\mathbf{g}(r,s,\eta):=r\Lambda (M^{T}(\eta-s))^{-1}B^{*}(s)$, $\mathbf{K}(r,s,\eta):=re^{-s w}\widehat{p}+\tilde{V}^{w}(r,-s,\eta)$,  $\mathbf{g}(r,0,\eta)=r\Lambda(M^{T}(\eta))^{-1}=rA^{T}(\eta)$, $\mathbf{K}(r,0,\eta)=r\widehat{p}\neq \mathbf{0}$. Note that for $|r|<1$, $Re(\eta)\geq 0$, $|\mathbf{g}_{i,j}(r,0,\eta)|<1$, $|\mathbf{K}_{i,j}(r,0,\eta)|<1$, $i,j\in E$. Moreover, if $||x||_{\infty}=max_{1\leq i\leq N}\{x_{i}\}$ for a vector $\mathbf{x}=(x_{1},\ldots,x_{N})$, and for a matrix $M=(M_{i,j})_{1\leq i,j\leq N}$, let $||M||=max_{1\leq i,j\leq N}|M_{i,j}|$, then $||\mathbf{g}(r,0,\eta)||<1$,  $||\mathbf{K}(r,0,\eta)||<1$.

It would be also of high importance to consider recursions that result in the following matrix functional equations:
\begin{equation*}
        \mathbf{f}(r,s,\eta)=\mathbf{g}(r,\zeta(s),\eta)\mathbf{f}(r,\zeta(s),\eta)+\mathbf{K}(r,s,\eta).
    \end{equation*}
In our case $\zeta(s)=as$. It seems that for  general $\zeta(s)$, the analysis can be similarly handled when we ensure that $\zeta(s)$ is a contraction on $\{s\in\mathbb{C}:Re(s)\geq 0\}$.
\end{remark}
\begin{lemma}
\begin{enumerate}
    \item The $N$ eigenvalues, say $\nu_{i}$, $i=1\ldots,N$, of $\Lambda-Q^{T}$ all lie in $Re(s)>0$.
    \item The $N$ zeros of $det((\eta-s)I+\Lambda-Q^{T})=0$ for $Re(\eta)\geq 0$, say $\mu_{i}(\eta)$, $i=1,\ldots,N$, are all in $Re(s)>0$ (i.e., For $Re(s)=0$, $Re(\eta)\geq 0$, $det((\eta-s)I+\Lambda-Q^{T})\neq 0$), and such that $\mu_{i}(\eta)=\nu_{i}+\eta$, $i=1,\ldots,N$. 
\end{enumerate}
\end{lemma}
\begin{proof}
Note that $\Lambda-Q^{T}-s I:=\Phi+R(s)$, where $R(s)=diag(\lambda_{1}+q_{1}-s,\ldots,\lambda_{N}+q_{N}-s)$, and 
\begin{displaymath}
    \Phi:=(\phi_{i,j})_{i,j=1,\ldots,N}:=\begin{pmatrix}
0&q_{2,1}&\ldots&q_{N,1}\\
q_{1,2}&0&\ldots&q_{N,2}\\
\vdots&\vdots&\vdots\\
q_{1,N}&q_{2,N}&\ldots&0\\
\end{pmatrix}.
\end{displaymath}
Moreover,
\begin{displaymath}
    |\lambda_{i}+q_{i}-s|\geq \lambda_{i}+q_{i}-|s|>q_{i}=|-q_{i,i}|=\sum_{j\neq i}|q_{i,j}|=\sum_{j=1}^{N}|\phi_{i,j}|.
\end{displaymath}
Thus, from \cite[Theorem 1, Appendix 2]{smit}, for $Re(s)>0$, the number of zeros of $det(\Lambda-Q^{T}-s I)$ are equal to the number of zeros of $det(R(s))=\prod_{i=1}^{N}(\lambda_{i}+q_{i}-s)$. Assume now that all $\nu_{i}$ are distinct.

Similarly, for $Re(\eta)\geq 0$, $Re(s)=0$, $det(\Lambda-Q^{T}+(\eta-s) I)\neq 0$ (thus, the inverse of $M^{T}(\eta-s)$ exists), and all the zeros of $det(\Lambda-Q^{T}+(\eta-s) I)$ lie in $Re(s)>0$. It is easy to realize that these zeros, say $\mu_{i}(\eta)$, are such that $\mu_{i}(\eta)=\nu_{i}+\eta$, $i=1,\ldots,N$, where $\nu_{i}$ the eigenvalues of $\Lambda-Q^{T}$. 
\end{proof}

Thus,
\begin{displaymath}
    (M^{T}(\eta-s))^{-1}=\frac{1}{\prod_{i=1}^{N}(s-\mu_{i}(\eta))}L(\eta-s),
\end{displaymath}
where $L(\eta-s):=cof(M^{T}(\eta-s))$ is the cofactor matrix of $M^{T}(\eta-s)$. 
\begin{remark}
    Note that $(M^{T}(\eta-s))^{-1}$ can also be written as (see \cite[equation (3.16)]{regte}):
\begin{displaymath}
    (M^{T}(\eta-s))^{-1}=R\ diag(\frac{1}{\mu_{1}(\eta)-s},\ldots,\frac{1}{\mu_{N}(\eta)-s})R^{-1},
\end{displaymath}
where $R$ is the matrix with the $i$th column, say $R_{i}$, $i=1,\ldots,N$, being the right eigenvector of $\Lambda-Q^{T}$ corresponding to the eigenvalue $\nu_{i}$.
\end{remark}

    Let $F(\eta,s):=\Lambda L(\eta-s)B^{*}(s)$, and $F_{i,j}(\eta,s)$ the $(i,j)$ element of $F(\eta,s)$, $i,j=1,\ldots,N$. Then, \eqref{eq2} can be written as:
    \begin{equation}
        \prod_{i=1}^{N}\left(s-\mu_{i}(\eta)\right) \left[Z_{j}^{w}(r,s,\eta)-re^{-s w}p_{j}\right] -r\sum_{k=1}^{N}Z_{k}^{w}(r,as,\eta)F_{k,j}(\eta,s)=\prod_{i=1}^{N}(s-\mu_{i}(\eta))V_{j}^{w}(r,s,\eta).\label{eq3s}
    \end{equation}
    Note that for $|r|<1$, $Re(\eta)\geq 0$, 
\begin{itemize}
\item The left-hand side of \eqref{eq3s} is analytic in $Re(s)>0$, continuous in $Re(s)\geq 0$, and it is also bounded.
    \item The right-hand side of \eqref{eq3s} is analytic in $Re(s)<0$, continuous in $Re(s)\leq 0$, and it is also bounded.
\end{itemize}
Thus, Liouville's theorem \cite[Th. 2.52]{tit}, implies that, in their respective half planes, both the left and the right hand side of \eqref{eq3s} can be rewritten as a polynomial of at most $N$th degree in $s$, dependent of $r$, $\eta$, for large $s$: For $|r|<1$, $Re(\eta)\geq 0$, $Re(s)\geq 0$:
\begin{equation}
        \prod_{i=1}^{N}(s-\mu_{i}(\eta)) \left[Z_{j}^{w}(r,s,\eta)-re^{-s w}p_{j}\right] -r\sum_{k=1}^{N}Z_{k}^{w}(r,as,\eta)F_{k,j}(\eta,s)=\sum_{l=0}^{N}s^{l}C_{l,j}^{w}(r,\eta).\label{eq3ss}
    \end{equation}
In matrix notation, \eqref{eq3ss} is rewritten as
\begin{equation}
   \prod_{i=1}^{N}(s-\mu_{i}(\eta)) \left[\tilde{Z}^{w}(r,s,\eta)-re^{-s w}\widehat{p}\right] -rF(\eta,s)\tilde{Z}^{w}(r,as,\eta)=\sum_{l=0}^{N}s^{l}C^{w}_{l}(r,\eta),\label{eq4}
\end{equation}
where $C^{w}_{l}(r,\eta):=(C^{w}_{l,1}(r,\eta),\ldots,C^{w}_{l,N}(r,\eta))^{T}$, $l=1,\ldots,N$.

    Note that for $s=0$, \eqref{eq3ss} (having in mind that $F_{i,j}(\eta,0)=A^{T}_{i,j}(\eta)$) yields
\begin{equation*}
    (-1)^{N}\prod_{i=1}^{N}\mu_{i}(\eta) \left[Z_{j}^{w}(r,0,\eta)-rp_{j}\right] -r\sum_{k=1}^{N}Z_{k}^{w}(r,0,\eta)A_{k,j}(\eta)=C^{w}_{0,j}(r,\eta).\label{al11}
\end{equation*}
However, from \eqref{eq3s}, for $s=0$, 
\begin{equation*}
    (-1)^{N}\prod_{i=1}^{N}\mu_{i}(\eta) \left[Z_{j}^{w}(r,0,\eta)-rp_{j}\right] -r\sum_{k=1}^{N}Z_{k}^{w}(r,0,\eta)A_{k,j}(\eta)=0,\label{sl2}
\end{equation*}
since $V^{w}_{j}(r,0,\eta)=0$, $j=1,\ldots,N$. Thus, $C_{0,j}^{w}(r,\eta)=0$, so that $C_{0}^{w}(r,\eta)=(C_{0,1}^{w}(r,\eta),\ldots,C_{0,N}^{w}(r,\eta))^{T}=\mathbf{0}$. By setting in \eqref{eq3ss} $s=\mu_{i}(\eta)$, $=1,\ldots,N$, we obtain, for $j=1,\ldots,N,$
\begin{equation}
    -r\sum_{k=1}^{N}Z_{k}^{w}(r,a\mu_{i}(\eta),\eta)F_{k,j}(\eta,\mu_{i}(\eta))=\sum_{l=1}^{N}(\mu_{i}(\eta))^{l}C_{l,j}^{w}(r,\eta).\label{eq3ssq}
\end{equation}
In matrix notation \eqref{eq3ssq} is rewritten as
\begin{equation}
    -rF(\eta,\mu_{i}(\eta))\tilde{Z}^{w}(r,a\mu_{i}(\eta),\eta)=\sum_{l=1}^{N}(\mu_{i}(\eta))^{l}C_{l}^{w}(r,\eta).\label{eqs}
\end{equation}

In matrix notation \eqref{eq3ss} is rewritten for $Re(s)\geq 0$, $Re(\eta)\geq 0$, $|r|<1$ as,
\begin{equation} 
    \tilde{Z}^{w}(r,s,\eta)=K(r,s,\eta)\tilde{Z}^{w}(r,as,\eta)+L^{w}(r,s,\eta),\label{basic}
\end{equation}
where
\begin{displaymath}
    \begin{array}{rl}
       K(r,s,\eta):=&r\Lambda (M^{T}(\eta-s))^{-1}B^{*}(s)=rA^{T}(\eta-s)B^{*}(s)=r\frac{F(\eta,s)}{\prod_{j=1}^{N}(s-\mu_{j}(\eta))}, \vspace{2mm}\\
        L^{w}(r,s,\eta):=  & re^{-s w}\widehat{p}+\frac{1}{\prod_{j=1}^{N}(s-\mu_{j}(\eta))}\sum_{l=1}^{N}s^{l}C_{l}^{w}(r,\eta).
    \end{array}
\end{displaymath}
Note that \eqref{basic} is the matrix analogue of the functional equation (1) in \cite{adan}. In the following, we show that by using a similar iterative approach we can solve it. Iteration of \eqref{basic} yields
\begin{equation}
    \tilde{Z}^{w}(r,s,\eta)=\sum_{n=0}^{\infty}\prod_{m=0}^{n-1}K(r,a^{m}s,\eta)L^{w}(r,a^{n}s,\eta)+\lim_{n\to\infty}r^{n}\prod_{m=0}^{n-1}A^{T}(\eta-a^{m}s)B^{*}(a^{m}s)\tilde{Z}^{w}(r,a^{n}s,\eta),\label{iter}
\end{equation}
where in the second term in \eqref{iter} we used the fact that $A^{T}(\eta-a^{m}s)=\Lambda(M^{T}(\eta-a^{m}s))^{-1}$. Note that the elements of the matrix $A^{T}(s)$, i.e., $A_{i,j}(s)$ are the LSTs of the interarrival times given that the background process make a transition from state $i$ to $j$ in that time interval. 

Note that as $m\to\infty$, $B^{*}(a^{m}s)\to I$, and $A^{T}(\eta-a^{m}s)\to A^{T}(\eta)=\Lambda(M^{T}(\eta))^{-1}$. Thus, $||B^{*}(a^{m}s)||\to 1$ as $m\to\infty$, and $||B^{*}(a^{m}s)||\leq 1$, for $Re(s)\geq 0$. Moreover, $||A^{T}(\eta-a^{m}s)||\to||A^{T}(\eta)||=\tau(\eta)<1$, as $m\to\infty$.  There exists an $L>0$, such that for $Re(\eta)>Re(sa^{m})$, $m\geq L$, $||A^{T}(\eta-a^{m}s)||\leq \tau_{m}(\eta)<\infty$. Therefore, for $Re(\eta)>Re(sa^{m})$, and $m\geq L$, we have:
\begin{displaymath}
    ||A^{T}(\eta-a^{m}s)B^{*}(a^{m}s)||\leq ||A^{T}(\eta-a^{m}s)||||B^{*}(a^{m}s)||\leq ||A^{T}(\eta-a^{m}s)|| =\tau_{m}(\eta),
\end{displaymath}
which implies that for $\tilde{\tau}(\eta)=max_{\{m=0,\ldots,n\}}\{\tau_{m}(\eta)\}$ there exists a constant $c_{s}<\infty$ such that $||\prod_{m=0}^{n}A^{T}(\eta-a^{m}s)B^{*}(a^{m}s)||\leq c_{s}\tilde{\tau}^{n}(\eta)$. Since $|r|<1$, \eqref{iter} gives,
\begin{equation}
    \tilde{Z}^{w}(r,s,\eta)=\sum_{n=0}^{\infty}\prod_{m=0}^{n-1}K(r,a^{m}s,\eta)L^{w}(r,a^{n}s,\eta).\label{solp}
\end{equation}
We still need to derive the remaining column vectors $C_{i}^{w}(r,\eta)$, $i=1,\ldots,N$. By setting $s=a\mu_{i}(\eta)$, $i=1,\ldots,N$, in \eqref{solp} we obtain:
\begin{equation}
    \tilde{Z}^{w}(r,a\mu_{i}(\eta),\eta)=\sum_{n=0}^{\infty}\prod_{m=0}^{n-1}K(r,a^{m+1}\mu_{i}(\eta),\eta)\left(re^{- a^{n+1}\mu_{i}(\eta)w}\widehat{p}+\frac{\sum_{l=1}^{N}(a^{n+1}\mu_{i}(\eta))^{l}C_{l}^{w}(r,\eta)}{\prod_{j=1}^{N}(\mu_{i}(\eta)a^{n+1}-\mu_{j}(\eta))}\right).\label{solp1}
\end{equation}
By substituting \eqref{solp1} in \eqref{eqs} we obtain a system of equations to obtain the remaining unknown vectors $C^{w}_{l}(r,\eta):=(C^{w}_{l,1}(r,\eta),\ldots,C^{w}_{l,N}(r,\eta))^{T}$, $l=1,\ldots,N$.
\begin{remark}
Note that in \eqref{solp}, $\tilde{Z}^{w}(r,s,\eta)$ appears to have singularities in $s=\mu_{j}(\eta)/a^{n}$, $n = 0,1,\ldots$, $j=1,\ldots,N$, $Re(\eta)\geq 0$, but
it can be seen that these are removable singularities. In the following, we show this for $s=\mu_{j}(\eta)$, $j=1,\ldots,N$. Note that \eqref{solp} is rewritten as
\begin{displaymath}
    \begin{array}{rl}
       \tilde{Z}^{w}(r,s,\eta)=  &\sum_{n=0}^{\infty}r^{n}\prod_{m=0}^{n-1}A^{T}(\eta-a^{m}s)B^{*}(a^{m}s)\left(re^{-s a^{n}w}\widehat{p}+\frac{\sum_{l=1}^{N}(a^{n}s)^{l}C_{l}^{w}(r,\eta)}{\prod_{j=1}^{N}(s a^{n}-\mu_{j}(\eta))}\right) \vspace{2mm}\\
         =& re^{-s w}\widehat{p}+\frac{\sum_{l=1}^{N}s^{l}C_{l}^{w}(r,\eta)}{\prod_{j=1}^{N}(s-\mu_{j}(\eta))}\vspace{2mm}\\&+rA^{T}(\eta-s)B^{*}(s)\sum_{n=1}^{\infty}r^{n-1}\prod_{m=1}^{n-1}A^{T}(\eta-a^{m}s)B^{*}(a^{m}s)\left(re^{-s a^{n}w}\widehat{p}+\frac{\sum_{l=1}^{N}(a^{n}s)^{l}C_{l}^{w}(r,\eta)}{\prod_{j=1}^{N}(s a^{n}-\mu_{j}(\eta))}\right)\vspace{2mm}\\
         =&re^{-s w}\widehat{p}+\frac{1}{\prod_{j=1}^{N}(s-\mu_{j}(\eta))}\left[\sum_{l=1}^{N}s^{l}C_{l}^{w}(r,\eta)+rF(\eta,s)\right.\vspace{2mm}\\
         &\left.\times\sum_{k=0}^{\infty}r^{k}\prod_{t=0}^{k-1}A^{T}(\eta-a^{t+1}s)B^{*}(a^{t+1}s)\left(re^{-s a^{k+1}w}\widehat{p}+\frac{\sum_{l=1}^{N}(a^{k+1}s)^{l}C_{l}^{w}(r,\eta)}{\prod_{j=1}^{N}(s a^{k+1}-\mu_{j}(\eta))}\right)\right]\vspace{2mm}\\
         =&re^{-s w}\widehat{p}+\frac{1}{\prod_{j=1}^{N}(s-\mu_{j}(\eta))} \left[\sum_{l=1}^{N}s^{l}C_{l}^{w}(r,\eta)+rF(\eta,s)\tilde{Z}^{w}(r,as,\eta)\right].
    \end{array}
\end{displaymath}

Note that the term in brackets in the last line vanishes for $s=\mu_{j}(\eta)$ (see \eqref{eq3ssq}). Thus, confirming that $s=\mu_{j}(\eta)$ is not a pole of $\tilde{Z}^{w}(r,s,\eta)$. Hence, the same holds for the expression of $\tilde{Z}^{w}(r,s,\eta)$ in \eqref{solp}. By using \eqref{eq4}, we can also show that $\tilde{Z}^{w}(r,s,\eta)$ has no singularity in $s=\mu_{j}(\eta)/a^{n}$, $n=1,2,\ldots$, $Re(\eta)\geq 0$, $j=1,\ldots,N$.
\end{remark}
\begin{remark}
    It seems that the analysis we followed above can be also applied whenever the matrix $K(r,s,\eta):=rG(\eta-s,s)$ ($G(s,\eta)$ is an $N\times N$ matrix with elements as given in \eqref{fun}), is such that $||G(\eta-a^{m}s,a^{m}s)||=\tau(\eta)<1$ for $Re(\eta)\geq 0$. Such a result is crucial in ensuring the convergence of the series and infinite products when we applied the iterating procedure. 
\end{remark}

\begin{remark}
An alternative way to derive the above result is given below. Note that, \eqref{eq2} can be rewritten for $|r|<1$, $Re(s)=0$, $Re(\eta)\geq 0$ as:
\begin{equation}
       \prod_{i=1}^{N}(s-\mu_{i}(\eta))[\tilde{Z}^{w}(r,s,\eta)-re^{-s w}\widehat{p}]-r\Lambda L(\eta-s)B^{*}(s)\tilde{Z}^{w}(r,as,\eta)=\prod_{i=1}^{N}(s-\mu_{i}(\eta))\tilde{V}^{w}(r,s,\eta).\label{eq3}
  \end{equation}
For $|r|<1$, $Re(\eta)\geq 0$, 
\begin{itemize}
\item The left-hand side of \eqref{eq3} is analytic in $Re(s)>0$, continuous in $Re(s)\geq 0$, and it is also bounded.
    \item The right-hand side of \eqref{eq3} is analytic in $Re(s)<0$, continuous in $Re(s)\leq 0$, and it is also bounded since $|E(e^{-s[a W_{n}+S_{n}-A_{n+1}]^{-}-\eta T_{n+1}}1_{\{Y_{n+1}=j\}}|W_{1}=w)|\leq 1$, $Re(s)\leq 0$, $Re(\eta)\geq 0$.
\end{itemize}
Thus, by analytic continuation we can define an entire function such that it is equal to the left-hand side of \eqref{eq3} for $Re(s)\geq 0$, and equal to the right-hand side of \eqref{eq3} for $Re(s)\leq 0$ (with $|r|<1$, $Re(\eta)\geq 0$). Hence, by (a variant of) Liouville's theorem \cite{paman} (for vector-valued functions; see also \cite[p. 81, Theorem 3.32]{rudin} or \cite[p. 232, Theorem 9.11.1]{die}, or \cite[p. 113, Theorem 3.12]{allan}) behaves as a polynomial of at most $N$th degree in $s$. Thus, for $Re(s)\geq 0$,
\begin{equation}
  \prod_{j=1}^{N}(s-\mu_{j}(\eta))[\tilde{Z}^{w}(r,s,\eta)-re^{-s w}\widehat{p}]-r\Lambda L(\eta-s)B^{*}(s)\tilde{Z}^{w}(r,as,\eta)=\sum_{l=0}^{N}s^{l}C^{w}_{l}(r,\eta),\label{eq41}
\end{equation}
where $C^{w}_{l}(r,\eta):=(C^{w}_{l,1}(r,\eta),\ldots,C^{w}_{l,N}(r,\eta))^{T}$, $l=0,1,\ldots,N$, column vectors still have to be determined. Note that for $s=0$, \eqref{eq41} yields
\begin{equation}
   (-1)^{N}\prod_{j=1}^{N}\mu_{j}(\eta)[(I-r\Lambda(M^{T}(\eta))^{-1})\tilde{Z}^{w}(r,0,\eta)-r\widehat{p}]=C^{w}_{0}(r,\eta).\label{l11}
\end{equation}
However, from \eqref{eq3}, for $s=0$, 
\begin{equation}
    (-1)^{N}\prod_{j=1}^{N}\mu_{i}(\eta)[(I-r\Lambda(M^{T}(\eta))^{-1})\tilde{Z}^{w}(r,0,\eta)-r\widehat{p}]=\mathbf{0},\label{l2}
\end{equation}
since $\tilde{V}(r,0,\eta)=\mathbf{0}$, where $\mathbf{0}$, is $N\times 1$ column vector with all components equal to 0. Thus, $C^{w}_{0}(r,\eta)=\mathbf{0}$. 
\end{remark}

\subsection{A more general dependence structure}\label{sub}
Consider now the case where,
\begin{equation}
    E\left(e^{-s A_{n+1}}1_{\{Y_{n+1}=j\}}|S_{n}=t,Y_{n}=i\right)=\chi_{i,j}(s)e^{-\psi_{i}(s)t},\label{dep1}
\end{equation}
thus, the interarrival time depend also on the length of the previous service time. More precisely, the interarrival time $A_{n+1}$ consists of two parts: a component that depends on the previous service time and the state of the Markov chain at time $n$, presented by the term $e^{-\psi_{i}(s)t}$, and a ``standard" interarrival time that depends only on the state at time $n$ and $n+1$, i.e., with Laplace-Stieltjes transform $\chi_{i,j}(s)$, which does not depend on the service time. From the form of \eqref{dep1}, some additional assumptions have been made explicitly. In particular since for $s=0$, the sum of the expectations $\sum_{j=1}^{N}E(1_{\{Y_{n+1}=j\}}|S_{n}=t,Y_{n}=i)=e^{-\psi_{i}(0)}\sum_{j=1}^{N}\chi_{i,j}(0)=1$, we must have $\psi_{i}(0)=0$, and $\sum_{j=1}^{N}\chi_{i,j}(0)=1$, $i=1,\ldots,N$. Then, the bivariate LST:
\begin{displaymath}
    E\left(e^{-s A_{n+1}-\eta S_{n}}1_{\{Y_{n+1}=j\}}|Y_{n}=i\right)=\int_{0}^{\infty}e^{-\eta t}\chi_{i,j}(s)e^{-\psi_{i}(s)t}dB_{i}(t)=\chi_{i,j}(s)\beta_{i}^{*}(\eta+\psi_{i}(s)),
\end{displaymath}
with $Re(\eta+\psi_{i}(s))>0$. In order to derive explicit results (in terms of LST), we further assume that the functions $\chi_{i,j}(.)$, $\psi_{i}(.)$ are rational, i.e.,
\begin{displaymath}
    \chi_{i,j}(s):=\frac{P^{(1)}_{i,j}(s)}{\prod_{m=1}^{L_{i,j}}(s-\lambda^{(m)}_{i,j})},\,\psi_{i}(s):=\frac{P^{(2)}_{i}(s)}{\prod_{m=1}^{M_{i}}(s-\nu^{(m)}_{i})},\,i,j=1,\ldots,N,
\end{displaymath}
where $P_{i,j}^{(1)}(.)$ a polynomial of degree at
most $L_{i,j}-1$, not sharing the same zeros with the denominator of $\chi_{i,j}(s)$, and similarly, $P_{i}^{(2)}(.)$ polynomial of degree at most $M_{i}-1$, not sharing the same zeros with the
denominator of $\psi_{i}(s)$, for $i = 1,\ldots, N$. For convenience, assume that $Re(\nu_{i}^{(m)})<0$, $i=1,\ldots,N$, $m=1,\ldots,M_{i}$. 

Then, by repeating the steps in the proof of Theorem \ref{thr} we have:
\begin{theorem}\label{thr1}
    For $Re(s)=0$, $Re(\eta)\geq 0$, $|r|<1$, $j=1,2,\ldots,N$,
    \begin{equation}
        Z^{w}_{j}(r,s,\eta)-rp_{j}e^{-s w}=r\sum_{i=1}^{N}Z^{w}_{i}(r,as,\eta)\beta_{i}^{*}(s+\psi_{i}(\eta-s))\chi_{i,j}(\eta-s)+V_{j}^{w}(r,s,\eta),
        \label{eqx1z}
    \end{equation}
    or equivalently, in matrix notation,
    \begin{equation}
        \tilde{Z}^{w}(r,s,\eta)-rX(\eta-s)B^{*}(s,\eta)\tilde{Z}^{w}(r,as,\eta)=re^{-s w}\widehat{p}+\tilde{V}^{w}(r,s,\eta),\label{eq21}
    \end{equation}
    where $X(s)$ be an $N\times N$ matrix with $(i,j)$ element equal to $\chi_{i,j}(s)$, and $B^{*}(s,\eta)=diag(\beta_{1}^{*}(s+\psi_{1}(\eta-s)),\ldots,\beta_{N}^{*}(s+\psi_{N}(\eta-s)))$.
\end{theorem}

Note that \eqref{eqx1z} is rewritten for $j=1,\ldots,N,$ as:
\begin{equation}
    \begin{array}{l}
       \prod_{i=1}^{N}\prod_{m=1}^{L_{i,j}}(s-\lambda_{i,j}^{(m)}) \left[Z^{w}_{j}(r,s,\eta)-rp_{j}e^{-s w}\right] \vspace{2mm}\\-r\sum_{i=1}^{N}Z^{w}_{i}(r,as,\eta)P^{(1)}_{i,j}(s)\prod_{k\neq i}\prod_{m=1}^{L_{k,j}}(s-\lambda_{k,j}^{(m)})\beta_{i}^{*}(s+\psi_{i}(\eta-s))=\prod_{i=1}^{N}\prod_{m=1}^{L_{i,j}}(s-\lambda_{i,j}^{(m)})V_{j}^{w}(r,-s,\eta).
    \end{array}\label{hju}
\end{equation}
Note that for $|r|<1$, $Re(\eta)\geq 0$, 
\begin{itemize}
\item The left-hand side of \eqref{hju} is analytic in $Re(s)>0$, continuous in $Re(s)\geq 0$, and it is also bounded.
    \item The right-hand side of \eqref{hju} is analytic in $Re(s)<0$, continuous in $Re(s)\leq 0$, and it is also bounded.
\end{itemize}
Thus, by Liouville's theorem \cite[Th. 2.52]{tit}, for $Re(s)\geq 0$,
\begin{equation}
    \begin{array}{l}
       \prod_{i=1}^{N}\prod_{m=1}^{L_{i,j}}(s-\lambda_{i,j}^{(m)})\left[Z^{w}_{j}(r,s,\eta)-rp_{j}e^{-s w}\right] \vspace{2mm}\\-r\sum_{i=1}^{N}Z^{w}_{i}(r,as,\eta)P^{(1)}_{i,j}(s)\prod_{k\neq i}\prod_{m=1}^{L_{k,j}}(s-\lambda_{k,j}^{(m)})\beta_{i}^{*}(s+\psi_{i}(\eta-s))=\sum_{l=0}^{\sum_{k=1}^{N}L_{k,j}}C_{l,j}^{w}(r,\eta)s^{l}.
    \end{array}\label{hjuc}
\end{equation}
Using similar arguments as above, $C_{l}^{w}(r,\eta)=\mathbf{0}$. In matrix notation \eqref{hjuc} has for $Re(s)\geq 0$, $Re(\eta)\geq 0$, $|r|<1$, the same form as \eqref{basic},
where now
\begin{displaymath}
    \begin{array}{rl}
       K(r,s,\eta):=&rX(\eta-s)B^{*}(s,\eta), \vspace{2mm}\\
        L^{w}(r,s,\eta):=  & re^{-s w}\widehat{p}+\frac{1}{\prod_{i=1}^{N}\prod_{m=1}^{L_{i,j}}(s-\lambda_{i,j}^{(m)})}\tilde{g}(r,s,\eta),
    \end{array}
\end{displaymath}
where $\tilde{g}(r,s,\eta)=(\sum_{l=1}^{\sum_{k=1}^{N}L_{k,1}}C_{l,1}^{w}(r,\eta)s^{l},\ldots,\sum_{l=1}^{\sum_{k=1}^{N}L_{k,N}}C_{l,N}^{w}(r,\eta)s^{l})^{T}$. Iteration of the resulting equation yields:
\begin{equation}
    \tilde{Z}^{w}(r,s,\eta)=\sum_{n=0}^{\infty}\prod_{m=0}^{n-1}K(r,a^{m}s,\eta)L^{w}(r,a^{n}s,\eta)+\lim_{n\to\infty}r^{n}\prod_{m=0}^{n-1}X(\eta-a^{m}s)B^{*}(a^{m}s,\eta)\tilde{Z}^{w}(r,a^{n}s,\eta).\label{iterm}
\end{equation}
Note that as $m\to\infty$, $B^{*}(a^{m}s,\eta)\to B^{*}(0,\eta)=diag(\beta_{1}^{*}(\psi_{1}(\eta)),\ldots,\beta_{N}^{*}(\psi_{N}(\eta)))$. Thus, for $Re(\eta)>0$, and $m\to\infty$, $||B^{*}(a^{m}s,\eta)||\to||B^{*}(0,\eta)||=b(\eta)<1$. Moreover, $X(\eta-a^{m}s)\to X(\eta)$, where the $(i,j)$ elements of $X(\eta)$ is an LST of the distribution of a non-negative random variable. Thus, for $Re(\eta)>0$ and $m\to\infty$ $||X(\eta-a^{m}s)||\to ||X(\eta)||=y(\eta)<1$. Thus, there exists a $J>0$, such that for $m\geq J$, $Re(\eta)>0$, $||B^{*}(a^{m}s,\eta)||\leq \tilde{B}_{m}(\eta)<\infty$. Moreover, there exists a $K>0$, such that for $m\geq K$, $Re(\eta)>Re(a^{m}s)$, $||X(\eta-a^{m}s)||\leq \tilde{X}_{m}(\eta)<\infty$. If $\tilde{b}(\eta)=\max_{\{m=0,1,\ldots,N\}}\{\tilde{B}_{m}(\eta)\}$, $\tilde{x}(\eta)=\max_{\{m=0,1,\ldots,N\}}\{\tilde{X}_{m}(\eta)\}$, then, $||X(\eta-a^{m}s)B^{*}(a^{m}s,\eta)||\leq ||X(\eta-a^{m}s)||||B^{*}(a^{m}s,\eta)||\leq \tilde{x}(\eta)\tilde{b}(\eta)$. Therefore, there exists a constant $h_{s}$, such that $\prod_{m=0}^{n}X(\eta-a^{m}s)B^{*}(a^{m}s,\eta)\leq h_{s}[\tilde{x}(\eta)\tilde{b}(\eta)]^{n}$. Thus, \eqref{iterm} yields,
\begin{equation}
    \tilde{Z}^{w}(r,s,\eta)=\sum_{n=0}^{\infty}\prod_{m=0}^{n-1}K(r,a^{m}s,\eta)L^{w}(r,a^{n}s,\eta).\label{oi}
\end{equation}
The remaining coefficients $C_{l,j}^{w}(r,\eta)$, $l=1,\ldots,\sum_{k=1}^{N}L_{k,j}$, $j=1,\ldots,N,$ are derived similarly as above and further details are omitted.

\section{Two Markov-dependent models with a similar solution method}\label{related}
In this section, we present two models the analysis of which leads to a similar vector-valued fixed point functional equation  for the LST of the workload. More precisely, we focus on a modulated shot-noise queue with additional positive and negative jumps, as well as on a modulated single server queue with service time that is randomly dependent on the waiting time. 
\subsection{On a modulated shot-noise queue}\label{shot}
In this section, we investigate the workload at arrival instants of a modulated $D/G/1$ shot-noise queue, i.e., a queue where the server's speed is workload proportional: when the workload is $x$, the server's speed equals $rx$; see \cite{shots} for a recent survey on shot-noise queueing systems. Our system operates as follows: Assume that the interarrival times $A_{1},A_{2},\ldots$ of customers equal $t>0$, i.e., $P(A_{n}=t)=1$. There is a single server, and service requirements of successive customers $S_{1},S_{2},\ldots$ are i.i.d. random variables. We assume that just before the arrival of the nth customer, additional amount of work equal to
$C_n$ is added. This can be explained as random noise caused by the arrival, and may be positive or negative. Let
\begin{displaymath}
    C_{n}=\left\{\begin{array}{ll}
    C_{n}^{+},&\text{ with probability }p, \\
     -C_{n}^{-},&\text{ with probability }q=1-p. 
\end{array}\right.
\end{displaymath}
Assume now that for $n\geq 0$, $x,y\geq 0$, $i,j=1,\ldots,N$:
\begin{equation*}
    \begin{array}{l}
         P(C_{n+1}\leq x,S_{n}\leq y,Y_{n+1}=j|Y_{n}=i,(C_{r+1},S_{r},Y_{r}),r=0,1,\ldots,n-1) \vspace{2mm}\\
        =  P(C_{n+1}\leq x,S_{n}\leq y,Y_{n+1}=j|Y_{n}=i)=p_{i,j}F_{S,i}(y)G_{C,j}(x),
    \end{array}
\end{equation*}
where $C^{+}|j$ has a general distribution with LST $c_{j}^{*}(s)$, and $C^{-}|j\sim\exp(\nu_{j})$, $j=1,\ldots,N$. Then, if $W_{n}$ is the workload before the $n$th arrival, we are dealing with a modulated stochastic recursion of the form $W_{n+1}=[e^{-rA_{n+1}}(W_{n}+S_{n})+C_{n+1}]^{+}$. Then, for $j=1,\ldots,N,$
\begin{displaymath}
    \begin{array}{rl}
        Z_{j}^{n+1}(s)= &E\left(e^{-sW_{n+1}}1_{\{Y_{n+1}=j\}}\right)=\sum_{i=1}^{N}P(Y_{n}=i) p_{i,j}E\left(e^{-sW_{n+1}}|Y_{n+1}=j,Y_{n}=i\right) \vspace{2mm}\\
         =& \sum_{i=1}^{N}P(Y_{n}=i)p_{i,j} \left( pE\left(e^{-s(e^{-rt}(W_{n}+S_{n})+C^{+}_{n+1})}|Y_{n+1}=j,Y_{n}=i\right)\right.\vspace{2mm}\\
         &\left.+qE\left(e^{-s[e^{-rt}(W_{n}+S_{n})-C^{-}_{n+1}]^{+}}|Y_{n+1}=j,Y_{n}=i\right)\right)\vspace{2mm} \\
         =&\sum_{i=1}^{N}p_{i,j} \left[ pc_{j}^{*}(-s)\beta_{i}^{*}(se^{-rt})Z_{i}^{n}(se^{-rt})\right.\vspace{2mm}\\
         &\left.+qP(Y_{n}=i)E\left(\int_{y=0}^{e^{-rt}(W_{n}+S_{n})}e^{-s(e^{-rt}(W_{n}+S_{n})-y)}\nu_{j}e^{-\nu_{j}y}dy+\int_{y=e^{-rt}(W_{n}+S_{n})}^{\infty}\nu_{j}e^{-\nu_{j}y}dy|Y_{n}=i\right) \right]\vspace{2mm}\\
         =&\sum_{i=1}^{N}p_{i,j} \left[ pc_{j}^{*}(s)\beta_{i}^{*}(se^{-rt})Z_{i}^{n}(se^{-rt})\right.\vspace{2mm}\\
         &\left.+qP(Y_{n}=i)E\left(\frac{\nu_{j}}{\nu_{j}-s}e^{-se^{-rt}(W_{n}+S_{n})}-\frac{s}{\nu_{j}-s}e^{-\nu_{j}e^{-rt}(W_{n}+S_{n})}|Y_{n}=i\right)\right]\vspace{2mm}\\
         =&\left(pc_{j}^{*}(s)+q\frac{\nu_{j}}{\nu_{j}-s}\right)\sum_{i=1}^{N}p_{i,j}\beta_{i}^{*}(se^{-rt})Z_{i}^{n}(se^{-rt})-\frac{sq}{\nu_{j}-s}\sum_{i=1}^{N}p_{i,j}\beta_{i}^{*}(\nu_{j}e^{-rt})Z_{i}^{n}(\nu_{j}e^{-rt}).
    \end{array}
\end{displaymath}
Letting $n\to\infty$, so that $Z_{j}^{n}(s)\to Z_{j}(s)$ we have
\begin{equation}
    Z_{j}(s)=\left(pc_{j}^{*}(s)+q\frac{\nu_{j}}{\nu_{j}-s}\right)\sum_{i=1}^{N}p_{i,j}\beta_{i}^{*}(se^{-rt})Z_{i}(se^{-rt})-\frac{sq}{\nu_{j}-s}\sum_{i=1}^{N}p_{i,j}\beta_{i}^{*}(\nu_{j}e^{-rt})Z_{i}(\nu_{j}e^{-rt}).\label{ol}
\end{equation}
In matrix notation, \eqref{ol} is written as
\begin{equation}
    \tilde{Z}(s)=\tilde{C}(s)P^{T}B^{*}(se^{-rt})\tilde{Z}(se^{-rt})+\tilde{Q}(s),\label{ol1}
\end{equation}
where $\tilde{C}(s):=pC(s)+q\widehat{L}(s)$, $C(s):=diag(c_{1}^{*}(s),\ldots,c_{N}^{*}(s))$, $\widehat{L}(s):=diag(\frac{\nu_{1}}{\nu_{1}-s},\ldots,\frac{\nu_{N}}{\nu_{N}-s})$, $\tilde{Q}(s):=q(I-\tilde{L}(s))\tilde{r}$, $\tilde{r}=(r_{1},\ldots,r_{N})$, with $r_{j}=\sum_{i=1}^{N}p_{i,j}\beta_{i}^{*}(\nu_{j}e^{-rt})Z_{i}^{n}(\nu_{j}e^{-rt})$.

Note that \eqref{ol1} is of the form
\begin{equation}
        \tilde{Z}(s)=H(\zeta(s))\tilde{Z}(\zeta(s))+\tilde{Q}(s),\label{hui}
    \end{equation}
where $\zeta(s)=se^{-rt}$, $H(\zeta(s)):=\tilde{C}(s)P^{T}B^{*}(se^{-rt})$. Moreover, the $n$th iterative of $\zeta(s)$, i.e., $\zeta^{(n)}(s)=\zeta(\zeta(\ldots\zeta(s)\ldots))=se^{-rnt}\to 0$ as $n\to \infty$. Moreover, \eqref{hui} is slightly different from \eqref{aqw}, since the mapping $\zeta(s)$ is also on the coefficient matrix $H(.)$. However, $H(e^{-rtm}s)\to P^{T}$ as $m\to\infty$. Thus, we can apply Lemma \ref{lemma1} to study the convergence of the iterating procedure. Therefore, by iterating \eqref{hui}, and applying appropriately Lemma \ref{lemma1} we obtain:
\begin{equation}
    \tilde{Z}(s)=\sum_{n=0}^{\infty}\prod_{d=0}^{n-1}H(e^{-rtd}s)\Tilde{Q}(e^{-rnt}s)+\prod_{n=0}^{\infty}H(e^{-rtn}s)\tilde{\pi}.\label{axz2}
\end{equation}
The unknown vector $\tilde{r}$ can be derived similarly by following the steps in Proposition \ref{prop1}, so further details are omitted.
\begin{remark}
    The modulated $D/G/1$ shot-noise queue seems to be the only shot-noise queueing model that may lead to a functional equation for $\tilde{Z}(s)$ that can be treated in the way we did in this work. In particular, by taking $a=e^{-rt}$, the resulting recursion is $\tilde{Z}(s)=H(as)\tilde{Z}(as)+\tilde{Q}(s)$, and can be iteratively solved. The M/G/1 case is tractable, but the analysis seems to result in a differential equation for $\tilde{Z}(s)$, rather than in a functional equation of the type treated in this work. Any other interarrival time distribution seems to lead to
a challenging functional equation for $\tilde{Z}(s)$.
\end{remark}
\begin{remark}
    Note that the modulated recursion $W_{n+1}=[e^{-rA_{n+1}}(W_{n}+B_{n})+C_{n+1}]^{+}$, is a special case of the general modulated recursion of the form $W_{n+1}=[J_{n}(W_{n})+L_{n+1}]^{+}$, where $L_{n}$ maybe the difference of two positive random variables, and $J_{n}(t)$ a sequence of i.i.d. subordinators with $E(J_{n}(1))<1$, $n=1,2,\ldots$, i.e., $J_{n}(t)$ corresponds to a Levy thinning of $W_{n}$. The scalar version of this kind of processes was investigated recently in \cite{box2}. We plan to thoroughly investigate the vector-valued version of this multiplicative Lindley recursion in a future work. 
\end{remark}
\subsection{On a modulated single-server queue with service time randomly dependent on
the waiting time}\label{deps}
Consider now the following modulated version of a variant of a M/M/1 queue that was investigated in \cite[Section 5]{boxman}. For reasons that will become clear in the following (see Reamrk \ref{rem25}) assume that for $n\geq 0$, $x,y\geq 0$, $i,j=1,\ldots,N$:
\begin{equation*}
    \begin{array}{l}
         P(A_{n+1}\leq x,S_{n}\leq y,Y_{n+1}=j|Y_{n}=i,A_{2},\ldots,A_{n},S_{1},\ldots,S_{n-1},Y_{1},\ldots,Y_{n-1}) \vspace{2mm}\\
        =  P(A_{n+1}\leq x,S_{n}\leq y,Y_{n+1}=j|Y_{n}=i)=p_{i,j}(1-e^{-\mu y})(1-e^{-\lambda_{j}x}).
    \end{array}
\end{equation*}
Assume also that if the waiting time of the $n$th arriving
customer equals $W_n$, then her service time equals $[S_{n}-cW_{n}]^{+}$, where $c>0$. Using similar arguments as above:
\begin{displaymath}
    \begin{array}{rl}
        Z_{j}^{n+1}(s)= &E\left(e^{-sW_{n+1}}1_{\{Y_{n+1}=j\}}\right)= \sum_{i=1}^{N}P(Y_{n}=i)p_{i,j}E\left(e^{-s[W_{n}+[S_{n}-cW_{n}]^{+}-A_{n+1}]^{+}}|Y_{n+1}=j,Y_{n}=i\right)\vspace{2mm}\\
        =&\sum_{i=1}^{N}P(Y_{n}=i)p_{i,j}\left[E\left(\int_{0}^{W_{n}+[S_{n}-cW_{n}]^{+}}e^{-s(W_{n}+[S_{n}-cW_{n}]^{+}-y)}\lambda_{j}e^{-\lambda_{j}y}dy|Y_{n}=i\right)\right.\vspace{2mm}\\&\left.+E\left(\int_{W_{n}+[S_{n}-cW_{n}]^{+}}^{\infty}\lambda_{j}e^{-\lambda_{j}y}dy|Y_{n}=i\right)\right]\vspace{2mm}\\
         =&\sum_{i=1}^{N}P(Y_{n}=i)p_{i,j}E\left(\frac{\lambda_{j}e^{-s(W_{n}+[S_{n}-cW_{n}]^{+})}-se^{-\lambda_{j}(W_{n}+[S_{n}-cW_{n}]^{+})}}{\lambda_{j}-s}|Y_{n}=i\right)\vspace{2mm}\\
=&\sum_{i=1}^{N}P(Y_{n}=i)p_{i,j}E\left(\frac{\lambda_{j}}{\lambda_{j-s}}\left(\int_{cW_{n}}^{\infty}e^{-s(\bar{c}W_{n}+x)}\mu e^{-\mu x}dx+\int_{0}^{cW_{n}}e^{-sW_{n}}\mu e^{-\mu x}dx\right)\right.\vspace{2mm}\\
&\left.-\frac{s}{\lambda_{j}-s}\left(\int_{cW_{n}}^{\infty}e^{-\lambda_{j}(\bar{c}W_{n}+x)}\mu e^{-\mu x}dx+\int_{0}^{cW_{n}}e^{-\lambda_{j}W_{n}}\mu e^{-\mu x}dx\right)|Y_{n}=i\right)\vspace{2mm}\\

=&\sum_{i=1}^{N}P(Y_{n}=i)p_{i,j}E\left(\frac{\lambda_{j}}{\lambda_{j}-s}\left(e^{-sW_{n}}-\frac{s}{\mu+s}e^{-(s+\mu c)W_{n}}\right)-\frac{s}{\lambda_{j}-s}\left(e^{-\lambda_{j}W_{n}}-\frac{\lambda_{j}}{\mu+\lambda_{j}}e^{-(\lambda_{j}+\mu c)W_{n}}\right)|Y_{n}=i\right)\vspace{2mm}\\
=&\frac{\lambda_{j}}{\lambda_{j}-s}\sum_{i=1}^{N}p_{i,j}\left(Z_{i}^{n}(s)-\frac{s}{\mu+s}Z_{i}^{n}(s+\mu c)\right)-\frac{s}{\lambda_{j}-s}\sum_{i=1}^{N}p_{i,j}\left(Z_{i}^{n}(\lambda_{j})-\frac{\lambda_{j}}{\mu+\lambda_{j}}Z_{i}^{n}(\lambda_{j}+\mu c)\right).
    \end{array}
\end{displaymath}
As $n\to\infty$, $Z^{n}_{i}(s)\to Z_{i}(s)$, we have,
\begin{equation*}
    Z_{j}(s)=\frac{\lambda_{j}}{\lambda_{j}-s}\sum_{i=1}^{N}p_{i,j}\left[Z_{i}(s)-\frac{s}{\mu+s}Z_{i}(s+\mu c)\right]-\frac{s}{\lambda_{j}-s}\sum_{i=1}^{N}p_{i,j}\left[Z_{i}(\lambda_{j})-\frac{\lambda_{j}}{\mu+\lambda_{j}}Z_{i}(\lambda_{j}+\mu c)\right],
\end{equation*}
or equivalently,
\begin{equation}
\lambda_{j}\sum_{i=1}^{N}p_{i,j}Z_{i}(s)+(s-\lambda_{j})Z_{j}(s)-s\frac{\lambda_{j}}{\mu+s}\sum_{i=1}^{N}p_{i,j}Z_{i}(s+\mu c)=sv_{j}.
    \label{kop}
\end{equation}
Note that for $s=0$, \eqref{kop} yields $Z_{j}(0)=\sum_{i=1}^{N}p_{i,j}Z_{i}(0)$, thus, $\tilde{Z}(0)=\tilde{\pi}$. In matrix terms, \eqref{kop} is rewritten as:
\begin{equation}
    D(s)\tilde{Z}(s)=s\tilde{v}+\frac{s}{\mu+s}\Lambda P^{T}\tilde{Z}(s+\mu c),
    \label{kop1}
\end{equation}
%\begin{equation}
 %   (I-L(s)P^{T})\tilde{Z}(s)=\frac{s}{\mu+s}L(s)P^{T}\tilde{Z}(s+\mu c)-s\Lambda^{-1}L(s)\tilde{v},
%\end{equation}
where $D(s)=sI-\Lambda(I-P^{T})$, $\tilde{v}=(v_{1},\ldots,v_{N})^{T}$, $v_{j}=\sum_{i=1}^{N}p_{i,j}\left[Z_{i}(\lambda_{j})-\frac{\lambda_{j}}{\mu+\lambda_{j}}Z_{i}(\lambda_{j}+\mu c)\right]$, $j=1,\ldots,N$. It is easy to realize that $v_{j}=P(W=0;1_{\{Y_{n+1}=j\}})$.
\begin{lemma}
    The matrix $\Lambda(I-P^{T})$ has exactly $N$ eigenvalues $\gamma_{i}$, $i=1,\ldots,N$, with $\gamma_{1}=0$, and $Re(\gamma_{i})>0$, $i=2,\ldots,N$. 
\end{lemma}
\begin{proof}
    Clearly, $s:=\gamma_{1}=0$ is a root of $det(D(s)=0)$, since $P$ is a stochastic matrix. By applying Gersgorin's circle theorem \cite[Th. 1, Section 10.6]{lanc}, every eigenvalue of $\Lambda(I-P^{T})$ lies in at least one of the disks
    \begin{displaymath}
        \{s:|s-\lambda_{i}(1-p_{i,i})|\leq \sum_{k\neq i}|\lambda_{i}p_{k,i}|=\lambda_{i}\sum_{k\neq i}p_{k,i}\}.
    \end{displaymath}
    Therefore, for each $i$, the real part of $\gamma_{i}$ is positive.
\end{proof}

Let 
\begin{displaymath}
    \zeta_{s}:=\{s:Re(s)\geq 0, Det(D(s))\neq 0\}.
\end{displaymath}
Then, for $s\in \zeta_{s}$,
\begin{equation}
    \tilde{Z}(s)=A(s)\tilde{v}+M(s)\tilde{Z}(s+\mu c),\label{cz}
\end{equation}
where $A(s):=sD^{-1}(s)$, $M(s):=\frac{1}{\mu+s}A(s)\Lambda P^{T}$. Iterating \eqref{cz} $K$ times we derive
\begin{equation}
    \tilde{Z}(s)=\sum_{k=0}^{K}\prod_{l=0}^{k-1}M(s+l\mu c)A(s+k\mu c)\tilde{v}+\prod_{l=0}^{K}M(s+l\mu c)\tilde{Z}(s+(K+1)\mu c),\,s\in\widehat{\zeta}_{s},\label{cz1}
\end{equation}
with $\widehat{\zeta}_{s}:=\{s:Re(s)\geq 0, s+l\mu c\in\zeta_{s},l=0,1,\ldots\}.$ Note that \eqref{cz} has the same form as in \cite[eq. (5.5)]{combe}, and it is critical to prove the convergence of \eqref{cz1} as $K\to\infty$. By appropriately modifying Lemma \ref{lemma1}, or by appropriately using \cite[Lemma 5.2.1]{combe} we can ensure the convergence as $K\to\infty$, and obtaining,
\begin{equation}
    \tilde{Z}(s)=\sum_{k=0}^{\infty}\prod_{l=0}^{k-1}M(s+l\mu c)A(s+k\mu c)\tilde{v}:=T(s)\tilde{v}.\label{no}
\end{equation}

Note that $D(s)$ is singular at the eigenvalues of $\Lambda(I-P^{T})$, i.e., $det(D(s))=0$ at $s=\gamma_{i}$, $i=1,\ldots,N$. However, $\tilde{Z}(s)$ is analytic in the half-plane $Re(s)\geq 0$, and thus, the vector $\tilde{v}$  will be derived such that the right hand side of \eqref{no} is finite at $s=\gamma_{i}$, $i=1,\ldots,N$. Divide \eqref{kop1} with $s$ and denote by $\tilde{y}_{i}$, the left (row) eigenvector of $\Lambda(I-P^{T})$, associated with the eigenvalue $\gamma_{i}$, $i=1,\ldots,N$. Then, \eqref{kop1} is written as
\begin{equation}
    \tilde{y}_{i}(1-\frac{\gamma_{i}}{s})\tilde{Z}(s)=\tilde{v}-\frac{1}{\mu+s}\tilde{y}_{i}\Lambda P^{T}\tilde{Z}(s+\mu c),\,i=1,\ldots,N,\,Re(s)\geq 0.\label{ftu}
\end{equation}

Letting $s=\gamma_{i}$, $i=1,\ldots,N$, and using \eqref{no}, we obtain $N$ equations for the derivation of the $N$ elements of $\tilde{v}$:
\begin{equation}
    \tilde{y}_{i}\tilde{v}=1_{\{i=1\}}+\frac{1}{\mu+\gamma_{i}}\tilde{y}_{i}\Lambda P^{T}T(\mu c+\gamma_{i})\tilde{v},\,i=1,\ldots,N.
\end{equation}
%
%
%To derive $\tilde{v}$ we need to rewrite \eqref{no} as
%\begin{equation}
%    D(s)\tilde{Z}(s)=T(s)\tilde{v},\label{xq}
%\end{equation}
%where $N(s):=sI$, $C(s):=-\frac{s}{\mu+s}\Lambda P^{T}$,
%\begin{displaymath}
%    T(s):=\sum_{k=0}^{\infty}N(s+k\mu c)\prod_{j=0}^{k-1}C(s+j\mu c)D^{-1}(s+(j+1)\mu c).
%\end{displaymath}
%Due to the form of $D(s)$, each left (row) eigenvector of $\Lambda(I-P^{T})$, say $\tilde{y}_{i}$, is also a left eigenvector of $D(s)$ associated with the eigenvalue $s-\gamma_{i}$. Thus, by left multiplying for $i=1,\ldots,N$ \eqref{xq} and set $s=\gamma_{i}$, we obtain the following $N-1$ equations:
%\begin{equation}
%    0=\tilde{y}_{i}T(\gamma_{i})\tilde{v}.\label{q1}
%\end{equation}
\begin{remark}
    Note that for $c=0$, \eqref{kop} reduces to \cite[eq. (7)]{adan2}, in which we now have set $\tilde{G}_{i}(s)=\frac{\mu}{\mu+s}$, $i=1,\ldots,N$.
\end{remark}
\begin{remark}\label{rem25}
    Note that in case we allow the service time to depend on $Y_{n+1}=j$, then the functional equation would be much more complicated than the one in \eqref{cz}. In particular, we will have to cope with a functional equation of the following form:
    \begin{equation}
    \tilde{Z}(s)=A(s)\tilde{v}+G(s)\sum_{i=1}^{N}H^{(i)}(s)\tilde{Z}(s+\mu_{i} c),\label{czc1}
\end{equation}
where $G(s):=D^{-1}(s)\Lambda P^{T}$, and $H^{(i)}(s)$ is an $N\times N$ matrix with the $(i,i)$ element equal to $\frac{1}{\mu_{i}+s}$, $i=1,\ldots,N$, and all other elements equal to zero. We treat this kind of vector-valued functional equation with multiple recursive terms in an ongoing paper, since some additional technical requirements are needed \cite{dimir}. For the scalar version of such a kind of functional equation, see \cite{adan}.
\end{remark}

\section{Numerical examples}\label{num}
Consider the case where $N=2$. In the following, we assume exponentially distributed interarrival and service times, and we use the results derived in subsection \ref{xc}, and in Section \ref{fgm}. Our aim is to investigate the impact of auto-correlation and cross-correlation of the interarrivals and service times on the mean workload, as well as the impact of the dependence parameter $\theta$ in case of the additional dependence among interarrival and service times based on the FGM copula. The auto-correlation between $S_{m}$, $S_{m+n}$:
\begin{displaymath}
    \begin{array}{c}
      \rho(S_{m},S_{m+n})=\rho(S_{0},S_{n})=\frac{\sum_{i=1}^{N}\sum_{j=1}^{N}\pi_{i}(p_{i,j}^{(n)}-\pi_{j})\gamma_{i}\gamma_{j}}{\sum_{i=1}^{N}\pi_{i}s_{i}^{2}-(\sum_{i=1}^{N}\pi_{i}\gamma_{i})^{2}},\,n\geq 1,
    \end{array}
\end{displaymath}
with $\gamma_{i}$, $s_{i}^{2}$ be the mean and the second moment of the service time distribution, $p_{i,j}^{(n)}=P(Y_{n}=j|Y_{0}=i),\,i,j=1\geq 0,\,n\geq 0$. A similar expression can be derived for the auto-correlation between interarrival times, as well as for cross-correlation between $A_{n}$, $S_{n}$; see \cite{adan2, vlasioudep}. The effect of the autoregressive parameter $a$ on the number of product form terms that we need in order to derive the numerical results is also discussed. 

Numerical results show that we need a (relatively) small finite number of product form terms, i.e., iterations (depending on the values of the parameters and especially as $a$ decreases) in order to obtain the values of the unknown vector(s) $\tilde{v}$ (resp. $\tilde{v}^{(1)}$, $\tilde{v}^{(2)}$) in \eqref{lpo} (resp. in \eqref{lpo1}) ($\tilde{v}$ (resp. $\tilde{v}^{(1)}$, $\tilde{v}^{(2)}$) is (resp. are) unknown(s) in \eqref{siterr} (resp. in \eqref{siterr1}) and is (resp. are) derived as stated in Proposition \ref{prop1} (resp. Proposition \ref{prop2})), as well as to derive the Laplace-Stieltjes transform vector $\tilde{Z}(s)$ (resp. $\tilde{Z}(s;\theta)$). In doing that, we truncate the infinite sums of products at a specific value of $k$, such that the absolute maximum norm of the difference of two consecutive terms to be smaller than $10^{-7}$, e.g., for the derivation of the vectors $\tilde{v}^{(1)}$, $\tilde{v}^{(2)}$ using Proposition \ref{prop2}, let $u_{m,j}(n,k):=\prod_{d=0}^{k-1}U(a^{d+1}n\lambda_{j};\theta) \left( I-L_{m}(a^{k+1}n\lambda_{j}) \right)$, $n,m,j=1,2$, and $f_{j}(n,k):=\prod_{d=0}^{k}U(a^{d+1}n\lambda_{j};\theta)$. Then, we find the maximum $k_{m,j,n}$ such that $||u_{m,j}(n,k_{m,j,n})-u_{m,j}(n,k_{m,j,n}+1)||\leq 10^{-7}$, $n,m,j=1,2$, and the maximum $l_{j,n}$ such that $||f_{j}(n,l_{j,n})-f_{j}(n,l_{j,n}+1)||\leq 10^{-7}$, $n,j=1,2$. 
\paragraph{Example 1} Let $[\lambda_{1}\ \lambda_{2}]=[2\ 8]$, $[\mu_{1}\ \mu_{2}]=[4\ 10]\times\frac{1}{u}$, $u>0$. The parameter $u$ is used to explore the effect of increasing the expected service duration on the mean workload. We focus on two cases, i.e., case 1, where $P=\begin{pmatrix}
    0& 1\\
    1 &0
\end{pmatrix}$, and case 2, where $P=\begin{pmatrix}
    0.5& 0.5\\
    0.5 &0.5
\end{pmatrix}$. In such a scenario we have positive cross-correlation between the interarrival and service times. Note that in both cases $\pi_{1}=\pi_{2}=0.5$, although in case 1 we have autocorrelated service and interarrival times, while in case 2 there is no autocorrelation. Figure \ref{a1} shows that the mean workload is higher in the case with autocorrelation.
\begin{figure}[htp]
    \centering
    \includegraphics[scale=0.6]{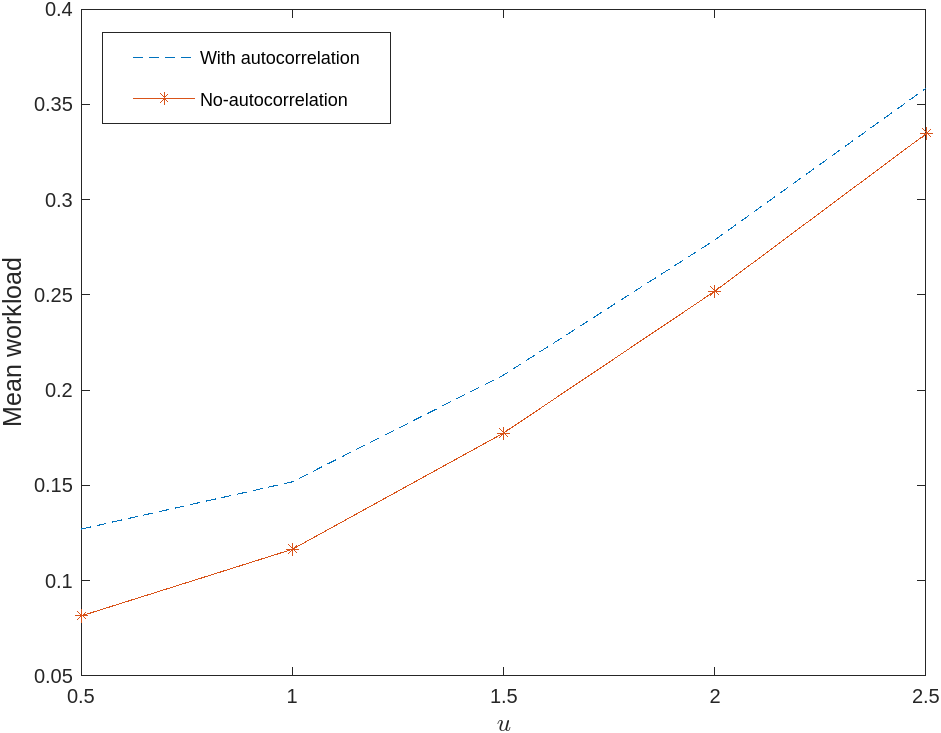}
    \caption{Positive cross-correlation: Mean workload as a function of parameter $u$ for $a=0.3$.}
    \label{a1}
\end{figure}

We now turn our attention to the case where there is an additional dependence among service times and interarrival times based on the FGM copula, by using the results derived in Section \ref{fgm} (in case there is no autocorrelation). We investigate first the impact of $a$ on the number of iterations that we need to obtain the numerical results. We observe that as we increase $a$ the number of product form terms needed in the summation \eqref{siterr1} to obtain the vectors $\tilde{v}^{(1)}$, $\tilde{v}^{(2)}$ increase very fast (see Tables \ref{tab1}, \ref{tab2}).
\begin{table}[!ht]
    \centering
    \begin{tabular}{||c| c| c| c| c| c| c| c|c||} 
 \hline
 $a$ & max $k_{1,1,1}$ & max $k_{1,2,1}$ & max $k_{2,1,1}$ & max $k_{2,2,1}$ & max $k_{1,1,2}$ & max $k_{1,2,2}$ & max $k_{2,1,2}$ & max $k_{2,2,2}$\\ [0.5ex] 
 \hline\hline
 0.1 & 7 & 7 & 8&8&8&7&7&9 \\ 
 \hline
 0.3 & 14 & 14& 15&15&14&13&13&15 \\
 \hline
 0.6 & 30& 30 & 32&32&34&29&29 &32\\
\hline
0.8& 67   & 67  &  60  &  60  &  76   & 71 &   48 &   48\\
  [1ex] 
 \hline
    \end{tabular}
    \caption{The effect of $a$ on the number of product form terms in the sum at the left hand side in \eqref{siterr1} for the case 2 when $u=2.5$.}
    \label{tab1}
\end{table}
\begin{table}[!ht]
    \centering
    \begin{tabular}{||c| c| c| c|c||} 
 \hline
 $a$ & max $l_{1,1}$ & max $l_{1,2}$ & max $l_{2,1}$ & max $l_{2,2}$ \\ [0.5ex] 
 \hline\hline
 0.1 & 7 & 7 & 8 &8\\ 
 \hline
 0.3 &13 &14& 14&15\\
 \hline
 0.6 & 30 & 31& 31 &31\\
 \hline
 0.8& 69 &   62 &   70 &   50\\
  [1ex] 
 \hline
    \end{tabular}
    \caption{The effect of $a$ on the number of product form terms in the second term at the left hand side in \eqref{siterr1} for the case 2 when $u=2.5$.}
    \label{tab2}
\end{table}

In Figure \ref{a11} we can see the impact of the additional dependence based on the FGM copula on the mean workload for $a=0.3$. We observed that the higher the dependence parameter $\theta$, the lower the mean workload. This is because, for example when the dependence parameter $\theta$ is negative, a long service time $S_{n}$ results in a short interarrival time $A_{n+1}$, which results in an increasing mean workload. Note that for $\theta=0$, i.e., the case where $F_{S,A|i,j}(y,x)=F_{S,j}(y)F_{A,i}(x)$, thus, $S_{n}$, $A_{n+1}$ are conditionally independent, given $Y_{n}$ and $Y_{n+1}$, the mean workload is between the curves for positive and negative $\theta$. Moreover, note that for $\theta=0$ the cross-correlation is equal to $0.1677$. For $\theta\in(0,1]$, the cross-correlation remains positive and increases (e.g., for $\theta=0.9$, the cross-correlation equals $0.3521$), but when $\theta$ becomes negative we may result in negative cross-correlation, e.g., for $\theta=-0.9$, the cross-correlation equals $-0.0168$, although for $\theta=-0.6$, the cross-correlation equals $0.0447$, while for $\theta=-0.8182$, cross-correlation almost vanishes. Thus, the dependence parameter $\theta$, may be used also as a tool for smoothing cross-correlation, for changing its sign or even as a tool for vanishing the cross-correlation. 
\begin{figure}[htp]
    \centering
    \includegraphics[scale=0.6]{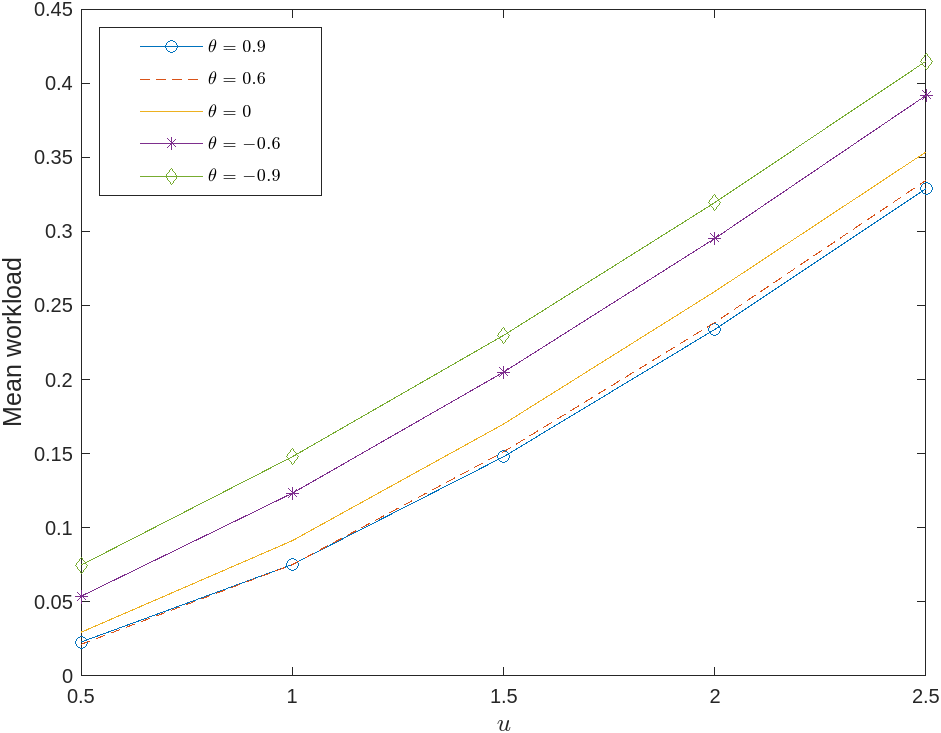}
    \caption{The impact of the dependence parameter $\theta$: Mean workload as a function of parameter $u$ for $a=0.3$.}
    \label{a11}
\end{figure}
\paragraph{Example 2}
Consider now the case where $[\lambda_{1}\ \lambda_{2}]=[2\ 8]$, $[\mu_{1}\ \mu_{2}]=[10\ 4]\times\frac{1}{u}$, $u>0$ and again explore the cases with auto-correlation (case 1) and no auto-correlation (case 2) considered above. In such a scenario, we have negative cross-correlation between the interarrival and service times. Figure \ref{a2} shows that the mean workload is slightly higher in the case with no auto-correlation. Clearly, as in example 1, similar observations can be made regarding the effect of the dependence parameter $\theta$, as well as the effect of $a$ on the number of iterations we need to obtain the numerical results.
\begin{figure}[htp]
   \centering
   \includegraphics[scale=0.6]{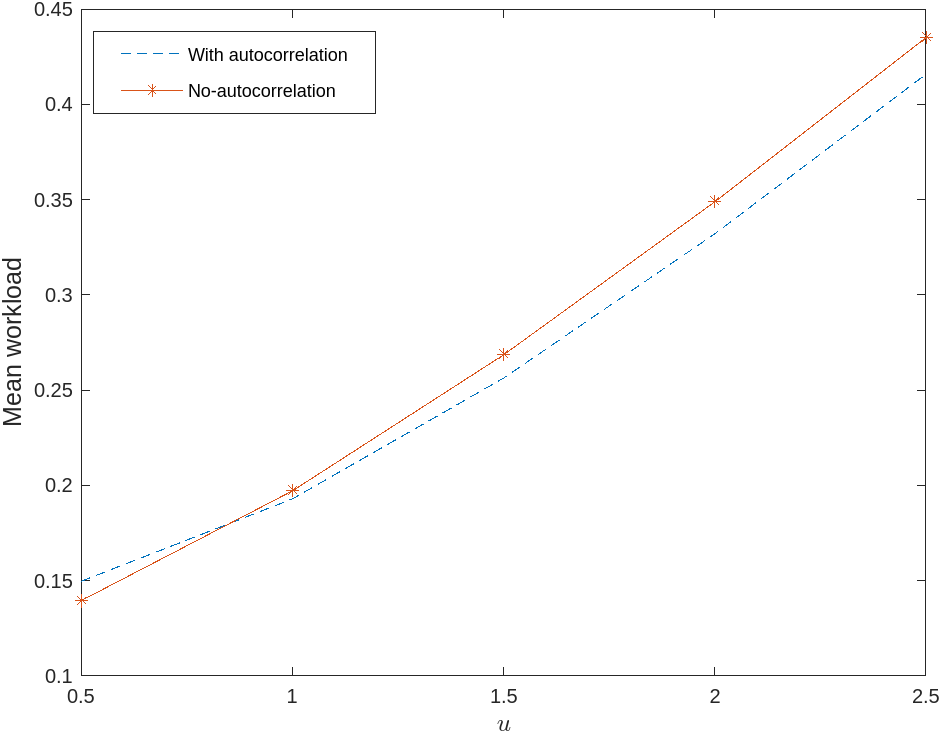}
   \caption{Negative cross-correlation: Mean workload as a function of parameter $u$.}
 \label{a2}
\end{figure}
\section{Concluding remarks and suggestions for future research}\label{con}
In this paper, strongly motivated by the work in \cite{box1}, we made a first step towards the development of a method for treating modulated recursions among random variables that lead to vector-valued fixed point functional equations of the form in \eqref{ji1}, \eqref{ji2}. In doing this, we focused on the detailed analysis of vector-valued reflected autoregressive processes, as well as of related processes that leads to a vector-valued fixed point equations of the form in \eqref{ji1}, \eqref{ji2}. We derived the Laplace-Stieltjes transform vector of the steady-state workload by iteratively solving a fixed point functional equation. 

In particular, we provided stationary, as well as transient analysis for a series of processes that satisfy a recursion of the form $W_{n+1}=[aW_{n}+S_{n}-A_{n+1}]^{+}$, $a\in(0,1)$ for which the distributions of $S_{n},A_{n+1}$ depend on the states of a background Markov chain. We investigated the case where given the states of the Markov chain, $S_{n},A_{n+1}$ are independent, as well as the case where they are dependent based either on FGM copula, or on a bilateral matrix exponential distribution. To our best knowledge, there are no analytic results referring to the derivation of the Laplace-Stieltjes transform of the steady-state workload in a Markov-dependent single server queue where interarrival times and service times are also dependent based on a copula. Moreover, we further investigated two related Markov-dependent models, the analysis of which leads to a similar vector-valued fixed point functional equation.

From technical point of view, by adopting the FGM copula, on one hand, we consider additional dependence among the service time and the subsequent interarrival time, and thus, provided more realistic scenarios, since assuming that service and interarrival times are conditionally independent (given the state of the background Markov chain) is restrictive. On the other hand, the number of unknown terms when solving \eqref{gh} is now doubled (we now have to derive $\tilde{v}^{(k)}$, $k=1,2$ instead of $\tilde{v}$ in the case we drop FGM copula), thus, it increases the computational effort. We also numerically illustrated the effect of the dependence via FGM copula on the mean workload. We observed that the higher the value of the dependence parameter, the lower the mean workload. Moreover, it seems that the dependence parameter can be used as a tool for smoothing cross-correlation, or for changing its sign (from positive to negative and vise-versa), e.g., in the case of exponentially distributed interarrival and service times, we noted that cross-correlation is a linear function of the dependence parameter $\theta$ (see Example 1 in Section \ref{num}). Of particular interest was the comparison of the cases $\theta=0$ (independent copula) and $\theta\in[-1,0)\cup(0,1]$. In example 1, for $\theta=0$, the cross-correlation is equal to $0.1677$. However, by setting $\theta=-0.9$, the cross-correlation drops to $-0.0168$. Moreover, for $\theta=-0.8182$, the cross-correlation almost vanishes. More work is needed on the investigation of the effect of the dependence parameter on the mean workload, as well as on its relation with cross-correlation, and is postponed as a future work. Finally, we observed that as the autoregressive parameter $a$
increases, the number of iterations that we need to obtain the numerical results increases rapidly. %The latter case is described by a simple modulated version of the recursion $W_{n+1}=[J_{n}(W_{n})+L_{n+1}]^{+}$, where $L_{n+1}$ the difference of two positive random variables and $J_{n}(t)$ a sequence of subordinators with $E(J_{n}(1))<1$, $n=1,2,\ldots$. 

In a future work, we plan to thoroughly investigate vector-valued versions of $W_{n+1}=[J_{n}(W_{n})+L_{n+1}]^{+}$, and consider their possible connection with the corresponding integer-valued processes; see \cite{box2} for the scalar case. Moreover, it would be interesting to consider scaling limits and
asymptotics. More precisely, it would be of great interest to introduce appropriate scalings, and apply a similar diffusion
analysis for the vector-valued reflected autoregressive process as the one done for the scalar case in \cite{box1}. A quite challenging task would be to cope with the problem where the autoregressive parameter is either affected by the background Markov chain \cite{dimir}, or it is a random variable.

\section*{Acknowledgements} The author is grateful to the Editors and the anonymous Reviewers for the insightful remarks and the very careful reading, which helped to improve the original exposition. The author gratefully acknowledges the Empirikion Foundation, Athens, Greece (\href{https://www.empirikion.gr/}{www.empirikion.gr}) for the financial support of this work.
\bibliographystyle{abbrv}

\bibliography{mybibfile}
\end{document}